\documentclass[reqno,12pt]{amsart}

\usepackage{amssymb}
\usepackage{amsfonts}
\usepackage{amsmath}
\usepackage{amsthm}
\usepackage{graphicx}
\usepackage{xcolor}
\usepackage[all]{xy}

\usepackage{hyperref}

\newcommand\sign{\operatorname{sign}}

\newcommand{\calC}{{\mathcal C}}

\newcommand{\sig}{\operatorname{sig}}
\newcommand{\odeg}{\operatorname{odeg}}

\newcommand{\marginnote}[1]{\ifthenelse{\isodd{\thepage}}{\normalmarginpar}
{\reversemarginpar}\marginpar{\fbox{\parbox{18mm}{\sloppy\footnotesize #1}}}}

\newtheorem{definition}{Definition}

\newtheorem{theorem}{Theorem}
\newtheorem{proposition}{Proposition}

\newtheorem{example}{Example}
\newtheorem{remark}{Remark}

\begin{document}

\title[Centrality measures in simplicial complexes]{Centrality measures in Simplicial Complexes: Applications of Topological Data Analysis  to Network Science.}

\date{\today}
\thanks{This work has been supported by Ministerio de Econom\'ia y Competitividad (Spain) and the European Union through FEDER funds under grants TIN2017- 84844-C2-1-R and MTM2017-86042-P, and the project STAMGAD 18.J445 / 463AC03 supported by Consejer\'{\i}a de Educaci\'on (GIR, Junta de Castilla y Le\'on, Spain).}



\author[D. Hern\'andez Serrano]{Daniel Hern\'andez Serrano}
\author[D. S\'anchez G\'omez]{Dar\'io S\'anchez G\'omez}

\address{Departamento de Matem\'aticas and Instituto Universitario de F\'isica Fundamental y Matem\'aticas (IUFFyM), Universidad de Salamanca,  Plaza de la Merced 1-4, 37008 Salamanca, Spain}
\email{dani@usal.es, dario@usal.es}

\begin{abstract}
Many real networks in social sciences, biological and biomedical sciences or computer science have an inherent structure of simplicial complexes reflecting many-body interactions. Therefore, to analyse topological and dynamical properties of simplicial complex networks centrality measures for simplices need to be proposed. Many of the classical complex networks centralities are based on the degree of a node, so in order to define degree centrality measures for simplices (which would characterise the relevance of a simplicial community in a simplicial network), a different definition of adjacency between simplices is required, since, contrarily to what happens in the vertex case (where there is only upper adjacency), simplices might also have other types of adjacency. The aim of these notes is threefold: first we will use the recently introduced notions of higher order simplicial degrees to propose new degree based centrality measures in simplicial complexes. These theoretical centrality measures, such as the simplicial degree centrality or the eigenvector centrality would allow not only to study the relevance of a simplicial community and the quality of its higher-order connections in a simplicial network, but also they might help to elucidate topological and dynamical properties of simplicial networks; sencond,  we define notions of walks and distances in simplicial complexes in order to study connectivity of simplicial networks and to generalise, to the simplicial case, the well known closeness and betweenness centralities (needed for instance to study the relevance of a simplicial community in terms of its ability of transmitting information); third, we propose a new clustering coefficient for simplices in a simplicial network, different from the one knows so far and which generalises the standard graph clustering of a vertex. This measure should be essential to know the density of a simplicial network in terms of its simplicial communities.
\end{abstract}

\maketitle
%{\small \tableofcontents}

%% main text

\section{Introduction} \label{intro}

Real world networks and data systems have been usually understood as graphs, where nodes represents the agents of the network and edges are drawn to keep track of the binary interactions between these agents. They present a complex behaviour, since we can not always infer global properties based on local ones, and Network Science and Statistical Mechanics of complex networks (\cite{BR02, BA16}) offer a universal language which allow to classify them, elucidate patterns of interactions and make predictions about the structure and evolution of such systems. But many real complex systems, such as social systems, data systems, brain networks or other biological complex systems, have a much richer inherent structure due to the fact that multi-interactions among agents are allowed. From this point of view, Topological Data Analysis (TDA) provides a mathematical tool (coming from algebraic topology) to deal with and faithfully represent multi-interactions among agents: the simplicial complexes. A simplex is just a set of different vertices and a simplicial complex is, roughly speaking, a collection $K$ of simplices which contains together with each simplex all its smaller subsets of vertices (called the faces of the simplex). Hence, simplices can be understood as higher dimensional generalisations of a point, line, triangle, tetrahedron, and so on, and they can codify multi interaction relations in classical networks (binary relations are lines or $1$-simplices, ternary relations are triangles or $2$-simplices, ...). The applications of TDA have been fruitful and varied during the last 15 years and have mainly (but not only) focused on the study of network topology and its evolution using homology and persistent homology of simplicial complexes; see for instance \cite{BASJK18,BK19,BC16,CH13,MDS15,MR12,MR14,PGV17} for applications in social systems, \cite{ER18,GH15,KFH16,XW14,XW15} for biomolecular systems, \cite{CL18,GGB15,GPCI15,LC19,P14,SPGB18} for brain networks, \cite{ME06} for network control or \cite{Ghrist08} for a survey on algebraic topology tools and data.

One of the most successful tools for analysing topological and dynamical properties of complex networks is the use of centrality measures: numeric quantifications of the relevance of a node attending to different parameters (such as position, number of connections, number of edges passing through it, number of links among its neighbours, ...). The concept of centrality has its origins in mathematical modeling of certain aspects and theories in sociology and psychology, and was introduced by Bavelas in \cite{Bav48} (which was  concerned with communication in small groups) and the notion of centrality index was formalised in \cite{Le51}. Their results show that centrality was well related to the efficiency of a group in problem-solving and the perception of the relevant agents in a network, thus opening the door to numerous applications and experiments in the following decades. A precise mathematical formulation of this framework was given in \cite{Sab66} and the ideas of Bavela's group have been used in several fields such as social studies (see for instance \cite{Free78} for a detailed review of the implications of the Bavela group's ideas and some of the first definitions of centrality measures in social networks). 

We aim to develop, from a theoretical  point of view, a mathematical framework in which to study centrality measures in simplicial complexes. If one attempts to define centrality measures based on a notion of degree for simplices (which would allow to characterise certain relevance of an agent, or group of agents, in a simplicial network), a different definition of adjacency between simplices is required, since, contrarily to what happens in the vertex case (where there is only upper adjacency between vertices connected by an edge), triangles, for example, have lower adjacent lines and vertices (at least all of its faces) and also might have upper adjacent simplices (if it is nested in a tetrahedron, for instance). In \cite{G02, MDS15, MR12,MR14,ER18} definitions of lower, upper and general adjacency for $q$-simplices  were given: two simplices $\sigma^{(q)}$ and $\sigma^{(q)}$ are lower adjacent if there exists a $(q-1)$-simplex $\tau^{(q-1)}$ which is a common $(q-1)$-face of both of them; they are said to be upper adjacent if there exists $\tau^{(q+1)}$ having both as $q$-faces; and they are considered adjacent if they are strictly lower adjacent but not upper adjacent. The associated degrees for these notions of adjacencies in simplicial complexes can be effectively computed by the entries of a matrix associated with a $q$-combinatorial Laplacian operator. Applications of these simplicial adjacencies and degrees can be found in \cite{G02, MR12, MR14} for theoretical or social-networks implications, \cite{ER18} for protein interactions or \cite{HJ13, MS13,PR17} for spectral theory and random walks on simplicial complexes.

But those definitions do not allow us to lower compare two simplices at different dimensions $q$ and $q'$ which share a $p$-face for general $p$, nor do they permit us to upper compare them if they are faces of a bigger $p'$-simplex. In \cite{HHS20} we have presented a mathematical framework generalising the notions of lower, upper and general adjacency and their associated degrees, which is valid for any simplicial dimension comparison: two simplices $\sigma^{(q)}$ and $\sigma^{(q')}$ are $p$-lower adjacent if there exists a $p$-simplex $\tau^{(p)}$ which is a common $p$-face of both of them; they are said to be $p$-upper adjacent if there exists $\tau^{(p)}$ having both as faces; and they are considered $p$-adjacent if they are strictly $p$-lower adjacent (meaning $p$-lower adjacent and not $(p+1)$-lower adjacent) but not $p'$-upper adjacent for a certain (explicitly given) dimension $p'$. The associated degrees to these notions of adjacencies are defined, and they are proved to generalise the usual ones. Moreover, a new higher order multi combinatorial Laplacian operator in an oriented simplicial complex is defined (by using a novel multi-parameter boundary operator) and it is used to effectively compute all the higher order degrees there defined. 

In this paper we apply the results given in \cite{HHS20} to introduce new centrality measures in simplicial complexes, which would allow to study the relevance of certain communities in a simplicial network and might be useful in studying topological and dynamical properties of complex systems. We start by defining centrality measures associated with the given notions of higher order degrees,  needed to understand the relevance and higher order relations among different collaborative simplicial communities and to study higher order degree distributions in simplicial complex networks. In fact, we have already revealed in \cite{HHS20} the potential practical utility of a simplicial degree centrality  measure (which we will define in these notes), since an structural analysis of higher-order connectivity of several real-world datasets (including coauthor networks, cosponsoring congress bills, contacts in schools, drug abuse warning networks, e-mail networks or publications and users in online forums) is performed is made there by studying statistical properties and the degree distributions associated with the new simplicial degrees. A rich and varied higher-order connectivity structures are shown to exists and real-world datasets of the same type reflect similar higher-order patterns. In addition, we will propose in this paper a generalized simplicial eigenvector centrality measure, extending in a natural way the one given in \cite{ER18}, and which allows to identify those $q$-simplices in a simplicial network that are p-adjacent to many other well-connected $q$-simplices. We also define the notions of walks and distances in simplicial complexes, different to that of \cite{MRV08,ER18}, to study connectivity and to generalise, to the simplicial case, the well-known closeness and betweenness centralities. These measures should be needed, for example, to study the relevance of a simplicial community in terms of its ability of transmitting information. Finally, we define a clustering coefficient for simplices in a simplicial complex different to that of \cite{MRV08}. This measure generalises the standard graph clustering of a vertex in a graph network, and allows to know not only the simplicial clustering around a node, but the clustering around a simplicial community. As far as a topological analysis of a simplicial network is concerned, this new simplicial clustering could be a useful tool to study a notion of simplicial density in a simplicial network and to reveal the existence and location of well-connected simplicial hubs (in classical Network Science, hubs are nodes with a high number of connections).

The paper is organised as follows. Section \ref{s:multi} is devoted to recall the basic definitions of simplicial complexes and the results of \cite{HHS20}. We start by giving the notions of lower, upper and general adjacency for simplices (valid for any dimensional comparison), its associated higher order degrees and how to effectively compute them by using the multi parameter boundary and combinatorial Laplacian operators.  The main results are tackled in Section \ref{s:cent}: centrality measures associated with the higher order simplicial degrees are introduced,  such as the simplicial degree centralities and the $p$-order simplicial eigenvector centrality; the notions of walks and distances in simplicial complexes are presented to study connectivity for simplicial complexes and to generalise, to the simplicial case, the closeness and betweenness centralities. Finally, we define in this section a clustering coefficient for simplices making use of both the higher order degree and the generalised distance in a simplicial complex. We end with Section \ref{s:conc} summarizing the main results and their potential practical applications within the simplicial network science.

\section{General adjacency, general simplicial degree and multi combinatorial Laplacian}\label{s:multi}\quad

We start by recalling the basic notions and properties of  simplicial complexes (we refer to  \cite{Mun84,G02} for a wide exposition and details), and then we will summarise the main results of \cite{HHS20}.

Given a finite set of points $\{v_0,v_1,\dots ,v_n\}$, which we call vertices, a $q$-simplex is a subset of vertices $\sigma^{(q)}=\{v_0,v_1,\dots ,v_q\}$ such that $v_i\neq v_j$ for all $i\neq j$ (where $q$ is referred to as the dimension of $\sigma^{(q)}$). A $p$-face (for $p\leq q$) of a $q$-simplex is just a subset $\{v_{i_0},\dots, v_{i_p}\}$ of the $q$-simplex. A simplicial complex $K$ is a collection of simplices such that if $\sigma$ is a simplex in $K$, then all the faces of $\sigma$ are also in $K$, and non-empty intersection of any two simplices of $K$ is a face of each of them. If a simplex is not a face of any other simplex, then it is called a facet. The dimension of a simplicial complex, $\dim K$, is the maximum among the dimensions of its simplices.

Recall that the degree of a vertex is the number of its incident edges, and this notion can be generalised to q-simplices. Notice that as a $0$-simplex, a vertex has degree $d$ if there are $d$ edges, $1$-simplices, incident to it, but a $1$-simplex have two $0$-simplices adjacent to it (the two vertices the edge has) but it also might be adjacent to a $2$-simplex (triangle) or to a $3$-simplex (tetrahedron). Then, notions of lower and upper adjacency (and its associated degrees) valid for any dimensional comparison between simplices are needed. These definitions are proposed in \cite{HHS20} and, since we will need them in order to define new centrality measures in simplicial complexes, we shall recall them later on in this section. But before getting into the simplicial case, let us recall that in graph network theory the degree of a vertex also appears as a diagonal entry of the graph Laplacian matrix, defined as $D-A$, where $D$ is a diagonal matrix with the degree of the vertices as diagonal entries, and $A$ is the usual vertex adjacency matrix. That is, one has an effective computational method to calculate the graph degree. In our general setting there also exists such a computational method by using the multi combinatorial Laplacian of \cite{HHS20} (which generalises the graph Laplacian and the $q$-combinatorial Laplacian). In order to define a combinatorial Laplacian, an orientation is required in the simplicial complex $K$: let $\sigma$ be a simplex in $K$, we define two orderings of its vertex set to be equivalent if they differ from one another by an even permutation. If $\dim\sigma >0$, this relations provides two equivalence classes and each of them is called an orientation of $\sigma$. An oriented simplex is a simplex $\sigma$ together with an orientation of $\sigma$. For a geometrically independent set of points $\{v_0, v_1,\dots, v_q\}$ we denote by $[v_0,\dots, v_q]$ and $-[v_0,\dots, v_q]$ the opposite oriented simplices spanned by $\{v_0, v_1,\dots, v_q\}$. We say that a finite simplicial complex $K$ is oriented if all of its simplices are oriented. We shall denote by $\widetilde{S}_p(K)$ the set of oriented $p$-simplices of the simplicial complex $K$, and by $S_p(K)$ the set of non oriented $p$-simplices.

Let us state few more notations and definitions. Given and oriented simplicial complex $K$, we define the  group of $q$-chains as the free abelian group $C_q(K)$ with basis the set of oriented $q$-simplices of $K$. The dimension $f_q$ of $C_q(K)$ is the number of $q$-dimensional simplices of the simplicial complex $K$, and it is codified in a topological invariant called the $f$-vector $f=(f_0,f_1,\dots ,f_q, \dots ,f_{\dim K})$. By assumption $C_q(K)$ is trivial if $q\notin[0,\dim K]$.

Let $C^q(K)$ the dual vector space of $C_q(K)$. Its elements, called \emph{cochains}, are completely determined by specifying its value on each simplex (since chains are linear combinations of simplices). Fixing an auxiliary inner product (with respect to which the basis of $C_q(K)$ can be chosen to be orthonormal), we can identify (via the associated polarity): $C_q(K)\simeq C^q(K)$.

Now we have the basic set up to recall the results of \cite{HHS20}.

\subsection{Higher order adjacency and simplicial degrees}\label{sec:higheradj}\quad 

Let us show how to lower and upper compare simplices of different dimensions. 

\begin{definition}\label{d:qhLoAdj} Lower adjacencies and lower degrees. \quad 
\begin{enumerate}
\item We say that $\sigma^{(q)}$ and $\sigma^{(q')}$ are $p$-lower adjacent if there exists a $p$-simplex $\tau^{(p)}$ which is a common face of both $\sigma^{(q)}$ and $\sigma^{(q')}$:
$$\sigma^{(q)}\sim_{L_p}\sigma^{(q')}\,\iff\, \exists\,\, \tau^{(p)} \quad \colon\quad  \tau^{(p)}\subseteq \sigma^{(q)}\quad \&\quad  \tau^{(p)}\subseteq \sigma^{(q')}\,.$$

Note that if $\sigma^{(q)}\sim_{L_p}\sigma^{(q')}$, then $\sigma^{(q)}\sim_{L_{p'}}\sigma^{(q')}$ for all $0\leq p'\leq p$. 

\item We say that $\sigma^{(q)}$ and $\sigma^{(q')}$ are strictly $p$-lower adjacent if: 
$$\sigma^{(q)}\sim_{L_{p^*}}\sigma^{(q')}\quad \iff \quad \sigma^{(q)}\sim_{L_p}\sigma^{(q')}\quad \mbox{ and }\quad \sigma^{(q)}\not \sim_{L_{p+1}}\sigma^{(q')}\,.$$ 

\item We define the $p$-lower degree of a $q$-simplex $\sigma^{(q)}$ as: 
$$\deg^{p}_L(\sigma^{(q)}):=\#\{\sigma^{(q')}\,\colon \, \sigma^{(q')}\sim_{L_p}\sigma^{(q)} \}\,.$$

\item We define the strictly $p$-lower degree of a $q$-simplex $\sigma^{(q)}$ as:
$$\deg^{p^*}_L(\sigma^{(q)}):=\#\{\sigma^{(q')}\,\colon\, \sigma^{(q')}\sim_{L_{p^*}}\sigma^{(q)} \}\,.$$

\item We define the $(h,p)$-lower degree of a $q$-simplex $\sigma^{(q)}$ as:
$$\deg^{h,p}_L(\sigma^{(q)}):=\#\{\sigma^{(q-h)}\,\colon\, \sigma^{(q-h)}\sim_{L_p}\sigma^{(q)}\}\,.$$

\item We define the strict $(h,p)$-lower degree of a $q$-simplex $\sigma^{(q)}$ as:
$$\deg^{h,p^*}_L(\sigma^{(q)}):=\#\{\sigma^{(q-h)}\,\colon\, \sigma^{(q-h)}\sim_{L_{p^*}}\sigma^{(q)}\}\,.$$

\end{enumerate}
\end{definition}

We have the following properties:
\begin{itemize}
\item Note that: 
$$\deg^{p}_L(\sigma^{(q)})=\displaystyle\sum_{h=q-\dim K}^{q-p}\deg^{h,p}_L(\sigma^{(q)})$$ and $$\deg^{p^*}_L(\sigma^{(q)})=\displaystyle\sum_{h=q-\dim K}^{q-p}\deg^{h,p^*}_L(\sigma^{(q)})\,.$$
\item If $h=1$ and $p=q-h=q-1$ then the $(1,q-1)$-lower degree of a $q$-simplex is the lower degree of the $q$-simplex of \cite{G02,ER18}:
\begin{align*}
\deg^{1,q-1}_L(\sigma^{(q)})&:=\#\{\tau^{(q-1)}\,|\, \tau^{(q-1)}\sim_{L_{q-1}}\sigma^{(q)}\}=\\
&=\#\{(q-1)-\mbox{faces of }\sigma^{(q)}\}=q+1
\end{align*}

\item If $p=q-h$ we have that:
\begin{equation}\label{e:lowfaces}
\begin{aligned}
\deg^{h,q-h}_L(\sigma^{(q)})&:=\#\{\tau^{(q-h)}\,|\, \tau^{(q-h)}\sim_{L_{q-h}}\sigma^{(q)}\}=\\
&=\#\{(q-h)-\mbox{simplices of }\sigma^{(q)}\}=\binom{q+1}{q-h+1}
\end{aligned}
\end{equation}

\item From the very definition we have that: 
\begin{equation}\label{e:stlowdeg}
\deg^{h,p^*}_L(\sigma^{(q)})=\deg^{h,p}_L(\sigma^{(q)})-\deg^{h,p+1}_L(\sigma^{(q)})\,.
\end{equation}
\end{itemize}

\begin{definition}\label{d:hpUdeg}Upper adjacencies and upper degrees. \quad 

\begin{enumerate}

\item We say that $\sigma^{(q)}$ and $\sigma^{(q')}$ are $p$-upper adjacent if there exists a $p$-simplex $\tau^{(p)}$ having both $\sigma^{(q)}$ and $\sigma^{(q')}$ as faces:
$$\sigma^{(q)}\sim_{U_p}\sigma^{(q')}\,\iff\, \exists\,\, \tau^{(p)} \quad \colon\quad  \sigma^{(q)}\subseteq \tau^{(p)} \quad \&\quad  \sigma^{(q')}\subseteq \tau^{(p)}\,.$$

\item We say that $\sigma^{(q)}$ and $\sigma^{(q')}$ are strictly $p$-upper adjacent if:
$$\sigma^{(q)}\sim_{U_{p^*}}\sigma^{(q')}\quad \iff \quad \sigma^{(q)}\sim_{U_p}\sigma^{(q')}\quad \mbox{ and }\quad \sigma^{(q)}\not \sim_{U_{p+1}}\sigma^{(q')}\,.$$ 

\item We define the $p$-upper degree of a $q$-simplex $\sigma^{(q)}$ as: 
$$\deg^{p}_U(\sigma^{(q)}):=\#\{\sigma^{(q')}\,\colon\, \sigma^{(q')}\sim_{U_p}\sigma^{(q)}\}\,.$$

\item We define the strictly $p$-upper degree of a $q$-simplex $\sigma^{(q)}$ as: 
$$\deg^{p^*}_U(\sigma^{(q)}):=\#\{\sigma^{(q')}\,\colon\, \sigma^{(q')} \sim_{U_{p^*}}\sigma^{(q)}\}\,.$$

\item We define the $(h,p)$-upper degree of a $q$-simplex $\sigma^{(q)}$ as:
$$\deg^{h,p}_U(\sigma^{(q)}):=\#\{\sigma^{(q+h)}\,\colon\, \sigma^{(q+h)}\sim_{U_p}\sigma^{(q)}\}\,.$$

\item We define the strict $(h,p)$-upper degree of a $q$-simplex $\sigma^{(q)}$ as:
$$\deg^{h,p^*}_U(\sigma^{(q)}):=\#\{\sigma^{(q+h)}\,\colon\, \sigma^{(q+h)}\sim_{U_{p^*}}\sigma^{(q)}\}\,.$$
\end{enumerate}
\end{definition}

We have the following properties and results:
\begin{itemize}
\item Note that: 
$$\deg^{p}_U(\sigma^{(q)})=\displaystyle\sum_{h=-q}^{p-q}\deg^{h,p}_U(\sigma^{(q)})$$ and $$\deg^{p^*}_U(\sigma^{(q)})=\displaystyle\sum_{h=-q}^{p-q}\deg^{h,p^*}_U(\sigma^{(q)})\,.$$
\item If $h=1$ and $p=q+h=q+1$ then the $(1,q+1)$-upper degree of a $q$-simplex is the upper degree of the $q$-simplex of \cite{G02,ER18}. 

\item For $q=0$, $h=1$ and $p=q+h=1$ then the $(1,1)$-upper degree of a vertex $v$ is the usual degree:
$$\deg^{1,1}_U(v)=\#\{\mbox {edges incident to }v\}=\deg (v)\,.$$

\item In general, for $p=q+h$, we have: 
$$\deg^{h,q+h}_U(\sigma^{(q)})=\#\{(q+h)\mbox {-simplices incident to }\sigma^{(q)}\}\,.$$

\item The strict upper degree can be computed as follows:
\begin{equation}\label{e:strictdegform}
\deg^{h,(q+h)^*}_U(\sigma^{(q)})=\sum_{i=0}^s (-1)^i \deg^{h+i,q+h+i}_U(\sigma^{(q)})\cdot \binom{h+i}{h}\,,
\end{equation}
where $s=\dim K -(q+h)$.
\end{itemize}

\begin{remark}\label{rem:Lp}
Let us point out some comments. 
\begin{enumerate}
\item As mentioned above, if $q=q'$ and $p=q-1$ (resp. $p=q+1$), then the notion $(q-1)$-lower (resp. $(q+1)$-upper) adjacency recovers the ordinary lower (resp. upper) adjacency for $q$-simplices of \cite{G02,ER18}. Thus, $(q+1)$-upper adjacency implies $(q-1)$-lower adjacency for $q$-simplices. However, in contrast to the ordinary case, the uniqueness of the common lower (resp. upper) simplex is no longer true in our general setting.

\item
If $h\geq 0$ and $\sigma^{(q)}\sim_{U_{q+h}}\sigma^{(q+h)}$, then $\sigma^{(q)}$ is a face of $\sigma^{(q+h)}$ and thus $\sigma^{(q)}\sim_{L_{q}}\sigma^{(q+h)}$. 
\item Although, in general is no longer true that $p$-upper adjacency implies $p'$-lower adjacency, one has that for $q\geq q'\geq h$ if $\sigma_i^{(q)}\sim_{U_{q+h}}\sigma_j^{(q')}$, then $\sigma_i^{(q)}\sim_{L_{q'-h}}\sigma_j^{(q')}$.
\end{enumerate}
\end{remark}

Notice that if $\sigma_i^{(q)}\sim_{L_{p^*}}\sigma_j^{(q')}$, then $\sigma_i^{(q)}$ and $\sigma_j^{(q')}$ share $(p+1)$-vertices. Thus, the smallest simplex which might contain both as faces (and therefore all of their vertices) has to have $q+1+q'+1-(p+1)=q+q'-p+1$ vertices, and thus it should be a $p'$-simplex where $p'=q+q'-p$. Moreover, in \cite{HHS20} it is proven that if  $\sigma_i^{(q)}\sim_{L_{p^*}}\sigma_j^{(q')}$ for some $p$ and $\sigma_i^{(q)}\not\sim_{U_{p'}}\sigma_j^{(q')}$, then $\sigma_i^{(q)}\not\sim_{U_{p'+h}}\sigma_j^{(q')}$ for all $h\geq 1$, where $p'=q+q'-p$. 

This justifies the following definitions. 

\begin{definition}\label{d:qhAdj}General adjacencies and their associated degrees. \quad 
\begin{enumerate}
\item We say that $\sigma^{(q)}$ and $\sigma^{(q')}$ are $p$-adjacent if they are strictly $p$-lower adjacent and not $p'$-upper adjacent, for $p'=q+q'-p$:
$$\sigma^{(q)}\sim_{A_p} \sigma^{(q')}\iff \sigma^{(q)}\sim_{L_{p^*}} \sigma^{(q')} \quad \& \quad \sigma^{(q)}\not\sim_{U_{p'}} \sigma^{(q')}.$$
In order to agree with graph theory, for $q=0$ we say that two vertices $v_i$ and $v_j$ are adjacent if $v_i\sim_{U_1} v_j$.

\item We say that $\sigma^{(q')}$ is maximal $p$-adjacent to $\sigma^{(q)}$ if:
\begin{align*}
\sigma^{(q')}\sim_{A_{p^*}} \sigma^{(q)} & \iff \sigma^{(q')}\sim_{A_p} \sigma^{(q)} \quad \&\\ & \quad \sigma^{(q')}\not \subset \sigma^{(q'')} \quad \forall\quad \sigma^{(q'')}\,|\, \sigma^{(q'')}\sim_{A_p} \sigma^{(q)}\,.
\end{align*}

\item We define the $p$-adjacency degree of a $q$-simplex $\sigma^{(q)}$ by:  
$$\deg^p_A(\sigma^{(q)}):=\#\{\sigma^{(q')} \,|\,\sigma^{(q)}\sim_{A_p} \sigma^{(q')}\}\,.$$

\item We define the maximal $p$-adjacency degree of a $q$-simplex $\sigma^{(q)}$ by:  
$$
\deg^{p*}_A(\sigma^{(q)}):=\#\{\sigma^{(q')} \,|\,\sigma^{(q')}\sim_{A_{p^*}} \sigma^{(q)}\}\,.
$$
\end{enumerate}
\end{definition}

\begin{figure}[!h]
\centering
\includegraphics[scale=0.7]{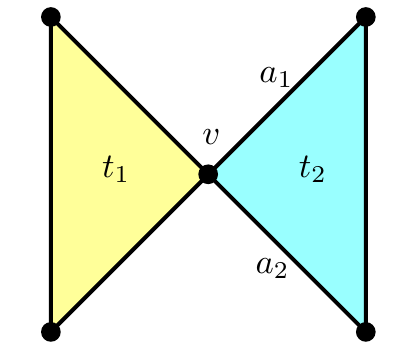}
\caption{Maximal $0$-adjacency.}
\label{fig:MaxAdj}
\end{figure}

Let us remark that with the $p$-adjacency degree we might be over counting certain simplices in the following sense: in Figure \ref{fig:MaxAdj} we have a triangle $t_1$ to which another triangle $t_2$ is $0$-adjacent in the vertex $v$, and, since there are two edges ($1$-faces) of $t_2$ that are $0$-adjacent to $t_1$ (the edges $a_1$ and $a_2$), they are also being counted in the $0$-adjacency degree of $t_1$. That is $\deg^0_A(t_1)=3$ because $t_2$, $a_1$ and $a_2$ are $0$-adjacent to $t_1$. Nonetheless, since $a_1$ and $a_2$ are $0$-adjacent to $t_1$, but they are also faces of another simplex $0$-adjacent to $t_1$ (they are faces of $t_2$), then, with the maximal $0$-adjacent degree of $t_1$ we are not counting $a_1$ and $a_2$, and thus $\deg^{0^*}_A(t_1)=1$. Thus, the maximal $p$-adjacency degree for a $q$-simplex only counts the maximal collaborative communities $p$-adjacent to a given simplex.

In order to count all the communities collaborating with a simplex (both the ones collaborating with  smaller sub-communities of the simplex and also the bigger ones where the simplex is nested in), a two parameter degree can be defined using both the $p_1$-adjacency degree and the $p_2$-upper degree as follows.

\begin{definition}\label{d:p1p2deg}
Given $p_1>q$ and $p_2<q$, we define the $(p_1,p_2^*)$-degree of a $q$-simplex $\sigma^{(q)}$ by:
$$\deg^{(p_1,p_2^*)}(\sigma^{(q)}):=\deg^{p_1}_U(\sigma^{(q)})+\deg^{p_2^*}_A(\sigma^{(q)})\,.$$
Similarly, for strict upper degree, we define the $(p_1^*,p_2^*)$-degree of a $q$-simplex $\sigma^{(q)}$ by:
$$\deg^{(p_1^*,p_2^*)}(\sigma^{(q)}):=\deg^{p_1^*}_U(\sigma^{(q)})+\deg^{p_2^*}_A(\sigma^{(q)})\,.$$
\end{definition}

Finally, we can get rid of the two parameters by defining a maximal simplicial degree of a $q$-simplex, which counts all the maximal communities collaborating with the sub-communities (faces) of the $q$-simplex (the ones that are maximal $p$-adjacent) and also the maximal communities on which the $q$-simplex is nested in (these ones being strictly upper adjacent).

\begin{definition}\label{d:simpdeg}
We define the maximal simplicial degree of $\sigma^{(q)}$ by:
$$\deg^*(\sigma^{(q)})=\deg^*_A(\sigma^{(q)})+\deg^*_U(\sigma^{(q)})\,,$$ 
where: 
$$\deg^*_A(\sigma^{(q)}):=\sum_{p=0}^{q-1}\deg^{p^*}_A(\sigma^{(q)})\,\text{ and }\,\deg^*_U(\sigma^{(q)}):=\sum_{h=1}^{\dim K-q}\deg^{h,(q+h)^*}_U(\sigma^{(q)})\,.$$
\end{definition}

\subsection{Generalized boundary operator and multi combinatorial Laplacian}\label{s:qhLap}\quad

Let us recall the generalised boundary operator on an oriented simplicial complex and the higher order multi combinatorial Laplacian, which will give us a way to effectively compute the higher order degrees of the previous subsection.

Let $\sigma^{(q)}$ be a $q$-simplex spanned by the set of points $\{v_0,\dots,v_q\}$. Let us denote $[v_0,\dots,\hat{v_i},\dots, v_q]$ the oriented $(q-1)$-simplex obtained from removing the vertex $v_i$ in $[v_0,\dots, v_q]$.

Given the $(q-h)$-face $\sigma^{(q-h)}$ of $\sigma^{(q)}$ spanned by $\{v_{0},\dots,\widehat v_{{j_1}},\dots,\widehat v_{{j_h}},\dots,v_{q}\}$ let us  denote by $\epsilon_{j_1\cdots j_h}$ the permutation 
$\begin{pmatrix}
0&\cdots& h-1&h&\cdots&q \\ 
{j_1}&\cdots&  {j_h}&0&\cdots & {q}
\end{pmatrix}$. As oriented $q$-simplex, $\sigma^{(q)}$ is represented by $[v_{\eta(0)},\dots,v_{\eta(q)}]$, for some permutation $\eta$ in the set of its vertices. 

\begin{definition}\label{d:qhBoundOp}
We define the {$(q,h)$-boundary operator} $$\partial_{q,h}\colon \calC_q(K)\to \calC_{q-h}(K)$$ as the homomorphism given as the linear extension of:

$$\partial_{q,h}([v_{\eta(0)},\dots, v_{\eta(q)}])=\sum_{j_{1},\dots,j_{h}} \sign(\eta)\sign(\epsilon_{j_1\cdots j_h})[v_0,\dots,\widehat{v_{j_1}},\dots, \widehat{v_{j_{h}}},\dots ,v_q]$$
where $[v_0,\dots,\widehat{v_{j_1}},\dots, \widehat{v_{j_{h}}},\dots ,v_q]$ denotes the oriented $q$-simplex obtained from removing the vertices $v_{j_1},\dots, v_{j_{h}}$ in $[v_0,\dots, v_q]$. 

\end{definition}
Note that $[v_{\eta(0)},\dots,v_{\eta(q)}]=[v_{\eta'(0)},\dots,v_{\eta'(q)}]$ if and only if $\sign (\eta)=\sign (\eta')$, so that $[v_{\eta(0)},\dots,v_{\eta(q)}]=\sign(\eta)\,[v_0,\dots,v_q]$. Then, this operator is well defined and $\partial_{q,h}(-\sigma^{(q)})=-\partial_{q,h}(\sigma^{(q)})$.
\begin{remark}
For $h=1$, the operator $\partial_{q,1}$ is the ordinary $q$-boundary operator $\partial_q\colon C_q(K)\to C_{q-1}(K)$ defined by (see for instance \cite{G02}): 
$$\partial_q([v_0,\dots, v_q])=\sum_{i=0}^q(-1)^i[v_0,\dots,\hat{v_i},\dots, v_q]\,.$$
\end{remark}

\begin{figure}[!h]
\centering
\includegraphics[scale=1.3]{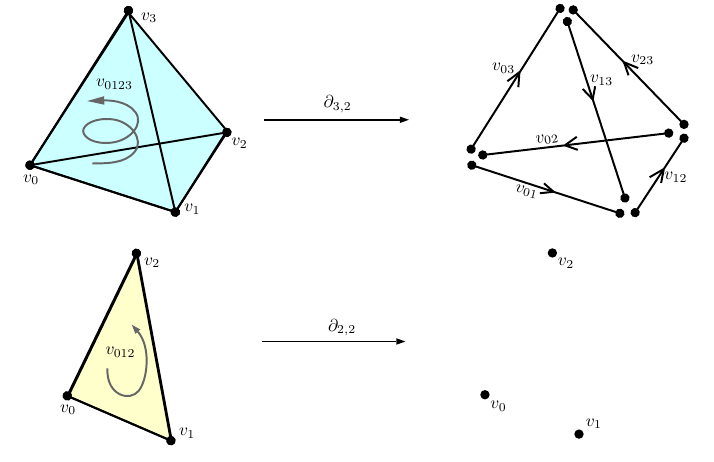}
\caption{Examples of $(q,h)$-boundary operators where we denote $[v_0,\dots, v_p]=v_{0\cdots p}$.}
\label{fig:Lqh}
\end{figure}

Given $\tau^{(p)}$ a $p$-simplex and $\sigma^{(q)}$ a $q$-face in $K$, with $q<p$, we denote by $\sign\big(\tau^{(p)},\sigma^{(q)}\big)$ the coefficient of $\sigma^{(q)}$ in the sum $\partial_{p,p-q}(\tau^{(p)})$.

\begin{definition}Upper similarity orientation, sign and upper oriented degree.\quad 
\begin{enumerate}
\item Let $\sigma_i^{(q)}$ and $\sigma_j^{(q')}$ be two simplices which are $p$-upper adjacent. Let $\tau^{(p)}$ be a common upper $p$-simplex. We say that $\sigma_i^{(q)}$ and $\sigma_j^{(q')}$ are {upper similarly oriented with respect to $\tau^{(p)}$} if $\sign\big(\tau^{(p)},\sigma^{(q)}\big)=\sign\big(\tau^{(p)},\sigma^{(q')}\big)$.

We shall denote it by $\sigma_i^{(q)}\sim_{U_{\tau^{(p)}}^+}\sigma_j^{(q')}$. If the signs are different, we say that they are { dissimilarly oriented} with respect to $\tau^{(p)}$. We shall denote it by $\sigma_i^{(q)}\sim_{U_{\tau^{(p)}}^-}\sigma_j^{(q')}$.

\item Let $\sigma_i^{(q)},\sigma_j^{(q')}$  and $\tau^{(p)}$ oriented simplices. We define the {upper sign} of $\sigma_i^{(q)}$ and $\sigma_j^{(q')}$ with respect to $\tau^{(p)}$ as the following function:
$$
\sig_U(\sigma_i^{(q)},\sigma_j^{(q')};\tau^{(p)}):=\begin{cases}
0 & \mbox{ if } \sigma_i^{(q)}\cup\sigma_j^{(q')}\nsubseteq\tau^{(p)}\\
1 & \mbox{ if } \sigma_i^{(q)}\sim_{U_{\tau^{(p)}}^+}\sigma_j^{(q')} 
\\
-1 & \mbox{ if } \sigma_i^{(q)}\sim_{U_{\tau^{(p)}}^-}\sigma_j^{(q')}
\end{cases}
$$

\item Let $\sigma_i^{(q)}$ and $\sigma_j^{(q')}$ two simplices. 

We define the $p$-{upper oriented degree} of $\sigma_i^{(q)}$ and $\sigma_j^{(q')}$ as the following sum:
$$\odeg^p_U(\sigma_i^{(q)},\sigma_j^{(q')}):=\frac{1}{2}\sum_{\tau^{(p)}\in \widetilde{S}_p(K)}\sig_U(\sigma_i^{(q)},\sigma_j^{(q')};\tau^{(p)})\,.$$
where by $\widetilde{S}_p(K)$ we denote the set of oriented $p$-simplices of the simplicial complex $K$. Note that we are dividing by $2$ in the oriented upper degree since we have $\sign\big(\sigma^{(q)},\tau^{(p)}\big)=\sign\big(\sigma^{(q)},-\tau^{(p)}\big)$.

\end{enumerate}

\end{definition}

Let us point out some comments and properties:
\begin{itemize}
\item For $p=q+1$ the first definition recovers the upper similarity of \cite{G02}.

\item Assume that $\tau^{(q)}$ is a $q$-simplex and $\sigma^{(p)}$ is a $p$-face of $\tau^{(q)}$, then:
$$\odeg_U^q(\sigma^{(p)},\tau^{(q)})=\sig_U(\sigma^{(p)},\tau^{(q)};\tau^{(q)})\,.$$
\end{itemize}

Let us now recall the analogous definitions for the lower adjacencies. 

\begin{definition} Lower similarity orientation, sign and lower oriented degree.\quad
\begin{enumerate}
\item Let $\sigma_i^{(q)}$ and $\sigma_j^{(q')}$ two simplices which are $p$-lower adjacent and  $\tau^{(p)}$ be a common lower $p$-face. We say that $\sigma_i^{(q)}$ and $\sigma_j^{(q')}$ are { lower similarly oriented with respect to $\tau^{(p)}$} if $\sign\big(\sigma^{(q)},\tau^{(p)}\big)=\sign\big(\sigma^{(q')},\tau^{(p)}\big)$.

If the signs are different, we say that they are {dissimilarly oriented} with respect to $\tau^{(p)}$. As before, we shall denote it by $\sigma_i^{(q)}\sim_{L_{\tau^{(p)}}^+}\sigma_j^{(q')}$ and $\sigma_i^{(q)}\sim_{L_{\tau^{(p)}}^-}\sigma_j^{(q')}$, respectively.

\item Let $\sigma_i^{(q)}$ and $\sigma_j^{(q')}$ two simplices. We define the {lower sign} of $\sigma_i^{(q)}$ and $\sigma_j^{(q')}$ with respect to a $p$-simplex $\tau^{(p)}$ as the following function:
$$
\sig_L(\sigma_i^{(q)},\sigma_j^{(q')};\tau^{(p)}):=\begin{cases}
0 & \mbox{ if } \tau^{(p)}\nsubseteq\sigma_i^{(q)}\cap\sigma_j^{(q')}\\
1 & \mbox{ if } \sigma_i^{(q)}\sim_{L_{\tau^{(p)}}^+}\sigma_j^{(q')} 
\\
-1 & \mbox{ if } \sigma_i^{(q)}\sim_{L_{\tau^{(p)}}^-}\sigma_j^{(q')} 
\end{cases}
$$

\item Let $\sigma_i^{(q)}$ and $\sigma_j^{(q')}$ two simplices. We define the $p$-{lower oriented degree} of $\sigma_i^{(q)}$ and $\sigma_j^{(q')}$ as the following sum:
$$\odeg^p_L(\sigma_i^{(q)},\sigma_j^{(q')}):=\frac{1}{2}\sum_{\tau^{(p)}\in \widetilde{S}_p(K)}\sig_L(\sigma_i^{(q)},\sigma_j^{(q')};\tau^{(p)})\,.$$
where by $\widetilde{S}_p(K)$ we denote the set of oriented $p$-simplices of the simplicial complex $K$. Note that we are dividing by $2$ since we have $\sign\big(\sigma^{(q)},\tau^{(p)}\big)=\sign\big(\sigma^{(q)},-\tau^{(p)}\big)$.

\end{enumerate}
\end{definition}

We have the following properties:
\begin{itemize}
\item For $p=q-1$ the first definition recovers the lower similarity notion of \cite{G02}.

\item Assume that $\tau^{(q)}$ is a $q$-simplex and $\sigma^{(p)}$ is a $p$-face of $\tau^{(q)}$, then:
$$\odeg_L^p(\sigma^{(p)},\tau^{(q)})=\sig_L(\sigma^{(p)},\tau^{(q)};\sigma^{(p)})$$
$$\sig_U(\sigma^{(p)},\tau^{(q)};\tau^{(q)})=\sign\big(\tau^{(q)},\sigma^{(p)}\big)=\sig_L(\sigma^{(p)},\tau^{(q)};\sigma^{(p)})\,.$$

\end{itemize}

With the above definitions and properties, the $(q,h)$-boundary operator can be rewritten by the following formula:
\begin{equation}\label{eq:boundary}
\partial_{q,h}(\tau^{(q)})=\displaystyle\sum_{\sigma^{(q-h)}\in S_{q-h}(K)}\odeg_L^{q-h}(\tau^{(q)},\sigma^{(q-h)})\sigma^{(q-h)}\,,
\end{equation}
and it can be shown that there exists a unique homomorphism: $$\partial_{q,h}^*\colon C_{q-h}(K)\to C_q(K)$$ 
defined as:
\begin{equation}\label{eq:adjoint}
\partial_{q,h}^*(\sigma^{(q-h)})=\displaystyle\sum_{\tau^{(q)}\in S_{q}(K)}\odeg_U^{q}(\sigma^{(q-h)},\tau^{(q)})\tau^{(q)},
\end{equation}
and such that $\partial_{_{q,h}}$ and $\partial_{q,h}^*$ are adjoint operators. For $h=1$ they are the usual $q$-boundary and $q$-coboundary operators.

Now we can define the multi combinatorial Laplacian. 

\begin{definition}
Let $q,h,h'$ non negative integers. We define the {$(q,h,h')$-Laplacian operator} $$\Delta_{q,h,h'}\colon C_q(K)\to C_q(K)$$ as the following operator:
$$\Delta_{q,h,h'}:=\partial_{q+h,h}\circ \partial^*_{q+h,h}+\partial^*_{q,h'}\circ \partial_{q,h'}\,.$$

The map $$\Delta^U_{q,h}=\partial_{q+h,h}\circ \partial^*_{q+h,h}\colon C_{q}(K)\to C_q(K)$$ is named the upper $(q,h)$-Laplacian operator and $$\Delta^L_{q,h'}=\partial^*_{q,h'}\circ \partial_{q,h'}\colon C_q(K)\to C_{q}(K)$$ is called th
 $(q,h')$-Laplacian operator. So that $\Delta_{q,h,h'}=\Delta^U_{q,h}+\Delta^L_{q,h'}$.
\end{definition}

\begin{remark}
For $h=h'=1$ we have that $\Delta_{q,1,1}$ is the ordinary $q$-Laplacian operator (see for instance \cite{G02}):
$$\Delta_q:=\partial_{q+1}\circ \partial_{q+1}^*+\partial_q^*\circ \partial_q\,.$$
\end{remark}

Let us fix basis of ordered simplices for $C_{q+h}(K),\, C_{q}(K)$ and $C_{q-h'}(K)$ and denote by $B_{q+h,h}$ and $B_{q,h'}$ the corresponding matrix representation of $\partial_{q+h,h}\colon C_{q+h}(K)\to C_{q}(K)$ and $\partial_{q,h'}\colon C_{q}(K)\to C_{q-h'}(K)$, respectively. Then the associated matrix of the $(q,h,h')$-Laplacian operator is 
$$L_{q,h,h'}=B_{q+h,h}\,B^t_{q+h,h}+B^t_{q,h'}\,B_{q,h'}\,.$$
We shall call it the $(q,h,h')$-Laplacian matrix and, as before, we use the notation  $L_{q,h}^U=B_{q+h,h}\,B^t_{q+h,h}$ and $L_{q,h'}^L=B^t_{q,h'}\,B_{q,h'}$.
\begin{theorem}(\cite{HHS20})\label{thm:Laplacian}
Let $K$ be an oriented simplicial complex and fix oriented basis on the $q$-chains $C_q(K)$ of $K$. With respect to this basis, the $(i,j)$-th entry of the associated matrices of the multi combinatorial Laplacian operators is given by:
\begin{align*}
\big(L^U_{q,h}\big)_{i,j}&=\begin{cases}
\deg_U^{h,q+h}(\sigma_i^{(q)}) & \mbox{ if } i=j \\
\odeg_U^{q+h}(\sigma_i^{(q)},\sigma_j^{(q)}) & \mbox{ if } i\neq j
\end{cases}\\
&\\
\big(L^L_{q,h'}\big)_{i,j}&=\begin{cases}
\deg_L^{h',q-h'}(\sigma_i^{(q)})=\binom{q+1}{q-h'+1} & \mbox{ if } i=j \\
\odeg_L^{q-h'}(\sigma_i^{(q)},\sigma_j^{(q)}) & \mbox{ if } i\neq j
\end{cases}\\
&\\
\big(L_{q,h,h'}\big)_{i,j}&=\begin{cases}
\deg_U^{h,q+h}(\sigma_i^{(q)})+\binom{q+1}{q-h'+1} & \mbox{ if } i=j \\
\odeg_U^{q+h}(\sigma_i^{(q)},\sigma_j^{(q)})+\odeg_L^{q-h'}(\sigma_i^{(q)},\sigma_j^{(q)}) & \mbox{ if } i\neq j
\end{cases}
\end{align*}
\end{theorem}

\subsection{How to explicitly compute higher order degrees}\quad

Notice that the multi combinatorial Laplacian does not compute all the higher order degrees of simplices, for instance the multi combinatorial Laplacian computes the lower degree $\deg_L^{h',q-h'}(\sigma^{(q)})$ of a simplex $\sigma^{(q)}$ (which we already knew to be equal to $\binom{q+1}{q-h'+1}$), but does not computes its general $p$-lower degree. Then, let us recall here the explicit description of the higher order degrees in terms of the sign functions and the matrices associated with the generalised boundary and coboundary operators given in \cite{HHS20}. 

We fix some notation. Consider $q,q'$ and $p$ non negative integers, denote $p'=q+q'-p$ and let
$\{\sigma_1^{(q')},\dots,\sigma_m^{(q')}\}$, $\{\sigma_1^{(q)},\dots,\sigma_n^{(q)}\}$, $\{\gamma_1^{(p)},\dots,\gamma_r^{(p)}\}$, $\{\gamma_1^{(p+1)},\dots,\gamma_s^{(p+1)}\}$ and $\{\tau_1^{(p')},\dots,\tau_t^{(p')}\}$ be basis of $C_{q'}(K),\, C_{q}(K)\,, C_{p}(K)$, $C_{p+1}(K)$ and $C_{p'}(K)$ respectively, then we have

\begin{theorem}\label{thm:degrees}
The higher order degrees of a $q$-simplex $\sigma^{(q)}$ can be explicitly computed as follows.
\begin{enumerate}
\item $\deg_L^p(\sigma_j^{(q)})= -1+\sum_{q'=p}^{\dim K}\sum_k\min\big(1,\sum_{i}|(B_{q,p})_{ij}||(B_{q',p})_{ik}|\big)\,.$
\item $\deg_U^p(\sigma_j^{(q)})= -1+\sum_{q'=0}^{p}\sum_k\min\big(1,\sum_{i}|(B_{p,p-q})_{ji}||(B_{p,p-q'})_{ki}|\big)\,.$
\item $\deg^p_A(\sigma_j^{(q)})=\displaystyle\sum_{q'=p}^{\dim K}\displaystyle\sum_{k=1}^{f_{q'}} adj^p(\sigma_j^{(q)},\sigma_k^{(q')})$, being 
$f_{q'}=\dim C_{q'}(K)$ and 
$$adj^p(\sigma^{(q)},\sigma^{(q')})= m_L(q,q';p))\big(1-m_L(q,q';p+1)\big)\big(1-m_U(q,q';p')\big)$$
with
\begin{align*}
m_L(q,q';p)=&\min\Big(1,\displaystyle\sum_{i} |\sig_L\big(\sigma^{(q)},\sigma^{(q')};\gamma_i^{(p)}|\big)\Big)\\ 
m_U(q,q';p)=&\min\Big(1,\displaystyle\sum_{i} |\sig_U\big(\sigma^{(q)},\sigma^{(q')};\gamma_i^{(p)}|\big)\Big)\,. 
\end{align*}

\item$\deg^{p^*}_A(\sigma_j^{(q)})=\deg^p_A(\sigma_j^{(q)})-\displaystyle\sum_{q'=p}^{\dim K}\displaystyle\sum_{k=1}^{f_{q'}}\Delta_{q',k}$, where 
{\small $$
\Delta_{q',k}=\min\Big(1,\sum_{q'',\ell}|\sig_L(\sigma^{(q')}_k,\sigma_\ell^{(q'')};\sigma_k^{(q')})|\cdot adj^p(\sigma_j^{(q)},\sigma_\ell^{(q'')})\Big)adj^p(\sigma_j^{(q)},\sigma_k^{(q')})
$$}for 
$$p\leq q'\leq \dim K,\,\, 1\leq k\leq f_{q'},\, q'+1\leq q''\leq\dim K,\, 1\leq \ell\leq f_{q''}$$ and being $\{\sigma_\ell^{(q'')}\}_{\ell}$ a basis of $C_{q''}(K)$.
\end{enumerate}

\end{theorem}

%\begin{remark}
%\rojo{
%Notice that although our definition of the multi combinatorial Laplacian generalises the notion of the ordinary Laplacian operator, there are a few remarks that we should mention. For instance, and as opposite with the ordinary Laplacian matrix, there might be null entries in the multi combinatorial Laplacian matrix coming not only from non adjacent simplices, but also for example from simplices which being lower adjacent have common faces with opposite orientations and then this cancels the corresponding oriented degree. On the other hand, and being aware that the study of the eigenvalues of the Laplacian operator is interesting by its connection with the combinatorial structure of the simplicial complex (\cite{HJ13}), we should notice that, unlike the ordinary Lapacian operator, for the multi combinatorial Laplacian an analysis of its eigenvalues is still unknown. Basically, this is because the composition of boundary operators, namely $\partial_{q,h}\circ\partial_{q+h',h'}$, does not vanish, as it holds for the standard case ($h=h'=1$). Consequently, at this time we have not suitable tools to derive relationships between eigenvalues or eigenvectors and centrality measures, in the spirit of the existing analyzes for the ordinary setting (see, for instance, \cite{Bon07}).
%}
%\end{remark}

\section{Centrality measures in simplicial complexes}\label{s:cent}\quad
 
This section is aimed to introduce new centrality measures in simplicial complexes. We start by defining centrality measures associated with the given notions of higher order degrees of  Section \ref{s:multi}. Two of them deserves a special mention: the maximal simplicial degree centrality of Definition \ref{d:simpdegcent} (since its relevance in the study of higher-order connectivity in simplicial real-world networks has been revealed in \cite{HHS20}) and the $p$-adjacency simplicial eigenvector centrality Definition \ref{d:psimpeigen} (since is a theoretical candidate to measure not only the quantity, but also the quality of the connections of a simplicial community). We also define walks and distances in simplicial complexes and study connectivity for simplicial complexes, which allow us to generalise for simplicial complexes the well-known closeness and betweenness centralities, needed to study the relevance of a simplicial community in terms of its ability in transmitting information. Finally, we define a simplicial clustering coefficient, different \textcolor{blue}{from} that of \cite{MRV08}, which allows to know the clustering not only around a node, but also around a simplicial community. 

\subsection{Degree centralties}\quad 

The degree centrality of a vertex in an usual graph network measures the popularity of the vertex in terms of the number of links it has, more specifically, it is defined as $C_{\deg}(v):=\frac{\deg (v)}{f_0-1}$, where $f_0$ is the number of vertices of the network (the first entry of the $f$-vector in a simplicial complex $K$). The numerator is just the number of links (pairwise interactions) the vertex has, and the normalization in the denominator consists of considering all possible links one could form with the vertex $v$ and the $f_0-1$ remaining vertices.  

Following this idea we give a first definition of a higher degree centrality for a vertex in a simplicial complex $K$, allowing us to keep track of the existing multi-interactions in a simplicial network. Notice that for vertices there is only upper adjacency. 

\begin{definition}\label{d:pdegcent} The $h$-upper degree centrality of a $0$-simplex $v$ in a simplicial complex $K$ is defined as the ratio:
$$C_{\deg^{h,h}_U}(v):=\frac{\deg_U^{h,h}(v)}{\binom{f_0-1}{h}}\,.$$
\end{definition}
The numerator is the number of $h$-simplices in $K$ to which $v$ belongs to, and the denominator is the number of possible $h$-simplices which could be formed with $v$ and the remaining $f_0-1$ vertices. That is, $C_{\deg^{h,h}_U}(v)$ is measuring a normalised popularity of the agent $v$ in terms of the communities of at least $h$ agents it belongs to. Notice that for $h=1$ it agrees with the usual degree centrality.

Let us point out that some of these $h$-communities might be nested in other higher dimensional ones, so that, if one wants to measure its popularity in terms of the communities of exactly $h$ agents not nested in higher dimensional ones, one should consider the strict upper degree of Definition \ref{d:hpUdeg}. Thus, we give a second definition.

\begin{definition}\label{d:p*degcent} The strict $h$-upper degree centrality of a $0$-simplex $v$ in a simplicial complex $K$ is defined to be the ratio:
$$C_{\deg^{h,h^*}_U}(v):=\frac{\deg_U^{h,h^*}(v)}{\binom{f_0-1}{h}}\,.$$
\end{definition}

Similarly we can define the $(h,q+h)$-upper degree centralities for $q$-simplices in $K$.

\begin{definition} The $(h,q+h)$-upper degree centrality of a $q$-simplex $\sigma^{(q)}$ in a simplicial complex $K$ is defined as the ratio:
$$C_{\deg^{h,q+h}_U}(\sigma^{(q)}):=\frac{\deg_U^{h,q+h}(\sigma^{(q)})}{\binom{f_0-(q+1)}{q+h}}\,,$$
where the numerator is the number of $(q+h)$-simplices which are $(q+h)$-upper adjacent to $\sigma^{(q)}$ and the denominator means all possible $(q+h)$-simplices one can construct using the $q+1$ vertices of $\sigma^{(q)}$ and the remaining $f_0-(q+1)$ ones. 

The strict $(h,q+h)$-upper degree centrality of a $q$-simplex $\sigma^{(q)}$ in a simplicial complex $K$ is defined as the ratio:
$$C_{\deg^{h,(q+h)^*}_U}(\sigma^{(q)}):=\frac{\deg_U^{h,(q+h)^*}(\sigma^{(q)})}{\binom{f_0-(q+1)}{q+h}}\,.$$
\end{definition}

Using equation (\ref{e:strictdegform}) we have:
$$C_{\deg^{h,(q+h)^*}_U}(\sigma^{(q)})=\frac{\sum_{i=0}^{\dim K-(q+h)} (-1)^i\cdot \binom{h+i}{h}\cdot \binom{f_0-(q+1)}{q+h+i}\cdot C_{\deg^{h+i,q+h+i}_U}(\sigma^{(q)})}{\binom{f_0-(q+1)}{q+h}}$$

Now, let us point out that for $q$-simplices we are measuring (with the above centralities) the collaborative communities of $q+h$ agents which $\sigma^{(q)}$ belongs to, but it could happen, for $q>0$, that some smaller sub-communities of $\sigma^{(q)}$ (some of its faces), might belong to other distinct collaborative simplicial communities, and such an inter-community collaboration can not be counted with the above definitions. This is due to the fact that, as opposite to the vertex case (where there is not lower adjacency), for $q>0$ there exists $p$-lower adjacency and, what is more important for this particular case, there exists a notion of  $p$-adjacency: different collaborative communities sharing a lower sub-community of strictly $p$ agents ($p$-face) and not being part of a common upper higher dimensional community (see Definition \ref{d:qhAdj}). Bearing in mind this situation, let use the $p$-adjacency degree of a $q$-simplex to define another centrality measure.

\begin{definition}\label{d:padjcentr}
The $p$-adjacency degree centrality of a $q$-simplex $\sigma^{(q)}$ in a simplicial complex $K$ is defined to be the ratio:
$$C_{\deg^{p}_A}(\sigma^{(q)}):=\frac{\deg_A^{p}(\sigma^{(q)})}{\binom{q+1}{p+1}\big(\sum_{q'=p+1}^{\dim K}\binom{f_0-(p+1)}{q'}-1\big)}\,.$$
The maximal $p$-adjacency degree centrality of a $q$-simplex $\sigma^{(q)}$ in a simplicial complex $K$ is defined as the ratio:
$$C_{\deg^{p*}_A}(\sigma^{(q)}):=\frac{\deg_A^{p*}(\sigma^{(q)})}{\binom{q+1}{p+1}\big(\sum_{q'=p+1}^{\dim K}\binom{f_0-(p+1)}{q'}-1\big)}\,.$$
\end{definition}
The normalisation in the denominator of the first formula means the following: for each $p$-face of $\sigma^{(q)}$ (there are $\binom{q+1}{p+1}$ ones) we consider all possible $q'$-simplices one could form with the $(p+1)$ vertices of the $p$-face joint with the $f_0-(p+1)$ remaining ones and, since we might be over counting $\sigma^{(q)}$ for each of its $p$-faces, we then remove one.

Notice that maximal $p$-adjacency degree centrality is measuring the popularity of a simplex in terms of the (normalised) number of communities in the network which are collaborating with the sub-communities (faces) of the given simplex, but it is not taking into account the maximal communities on which the simplex could be nested in. Therefore, we can try to use both the strictly $(h,q+h)$-upper degree centrality and the maximal $p$-adjacency degree centrality to define a new two parameter degree centrality by:

$$C_{\deg^{(p_1^*,p_2^*)}}(\sigma^{(q)}):=C_{\deg^{p_1^*}_U}(\sigma^{(q)})+C_{\deg^{p_2^*}_A}(\sigma^{(q)})\,,$$
but then, we will find ourselves with a measure which might be bigger than one. Thus, there are two possible solutions, both of them using the maximal simplicial degree of Definition \ref{d:simpdeg}. The first one is to define a simplicial centrality degree by the formula:
 $$\frac{\deg^*(\sigma^{(q)})}{\sum_{h=1}^{\dim K-q}\binom{f_0-(q+1)}{q+h}+\sum_{p=0}^{q-1}\binom{q+1}{p+1}\big(\sum_{q'=p+1}^{\dim K}\binom{f_0-(p+1)}{q'}-1\big)}\,,$$
 where $\deg^*(\sigma^{(q)})=\deg^*_A(\sigma^{(q)})+\deg^*_U(\sigma^{(q)})$ is the maximal simplicial degree of Definition \ref{d:simpdeg}. But, since in a simplicial complex $K$ there cannot be  more collaborations for a $q$-simplex (and its faces) than the total number of simplices minus $1$, we propose a simplification of this formula in the following definition. 

\begin{definition}\label{d:simpdegcent}
Let $\sigma^{(q)}$ be a $q$-simplex in $K$. We define the maximal simplicial centrality degree centrality of $\sigma^{(q)}$ as:
$$C_{\deg^{*}}(\sigma^{(q)}):=\frac{\deg^*(\sigma^{(q)})}{\sum_{i=0}^{\dim K}f_{i}-1}\,,$$
where $\deg^*(\sigma^{(q)})=\deg^*_A(\sigma^{(q)})+\deg^*_U(\sigma^{(q)})$ is the maximal simplicial degree of Definition \ref{d:simpdeg} and $f_{i}=\dim C_{i}(K)$.
\end{definition}

%Therefore, the maximal simplicial centrality degree of a $q$-simplex measures a normalised relevance of the $q$-simplex, keeping track of both the number of all collaborative simplicial communities with its faces (using the maximal $p$-adjacency degree) and the number of all simplicial communities to which the $q$-simplex belongs to (using the strictly $(q+h)$-upper degree).
%
%\begin{figure}[!h]
%\centering
%\includegraphics[scale=1.1]{AdjDeg.pdf}
%\caption{General simplicial degree of $1$-simplex.}
%\label{fig:AdjDeg}
%\end{figure}

Since the maximal simplicial degree of a $q$-simplex counts the number of distinct facets which the $p$-faces of the $q$-simplex belongs to (for $p<q$) and also counts the distinct facets (different from the above ones) which the $q$-simplex belong to, then the maximal simplicial centrality degree of a $q$-simplex measures a normalised relevance of the $q$-simplex, keeping track of both the number of all collaborative simplicial communities with its faces (using the maximal $p$-adjacency degree) and the number of all simplicial communities to which the $q$-simplex belongs to (using the strictly $(q+h)$-upper degree).  

This is what we propose as a normalised quantity of the relevance of a community in terms not of the sum of the relevance of its individuals, but of the sum of the (not repeated or overcounted) relevance of its $p$-subcommunites for all $p=0,\dots ,q$.

\begin{figure}[!h]
\centering
\includegraphics[scale=1.1]{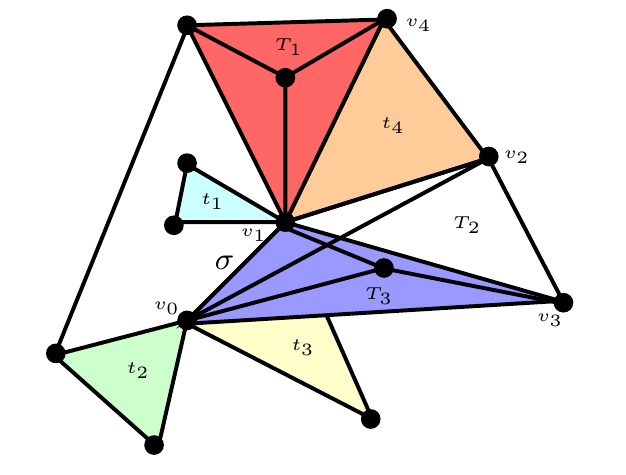}
\caption{General simplicial degree of $1$-simplex.}
\label{fig:AdjDeg}
\end{figure}

The practical usefulness of the maximal simplicial degree centrality  measure becomes apparent in \cite{HHS20}, where we have performed an structural analysis of higher-order connectivity of many real-world datasets by studying statistical properties and the maximal simplicial degree distribution, showing a rich and varied higher-order connectivity structures which, as one would expect, reflect similar higher-order patterns for datasets of the same type. Real-world datasets used there include coauthor networks, cosponsoring congress bills, contacts in schools, drug abuse warning networks, e-mail networks or publications and users in online forums.

\begin{example}\label{ex:AdjDeg}
Let us compute the maximal simplicial centrality degree centrality of the $1$-simplex $\sigma$ defined by the vertices $v_0$ and $v_1$ in Figure \ref{fig:AdjDeg}. We have $q=1$ and $\dim K=3$, so that:
\begin{align*}
\deg^*(\sigma)&=\deg^*_A(\sigma)+\deg^*_U(\sigma)=\\
&=\deg^{0^*}_A(\sigma)+\deg_U^{(1,2^*)}(\sigma)+\deg_U^{(2,3^*)}(\sigma)=\\
&=4+1+2=7\,,
\end{align*}
where:
\begin{itemize}
\item $\deg^{0^*}_A(\sigma)$ is the number of $q'$-simplices strictly $0$-lower adjacent to $\sigma$ and not upper adjacent. These are the triangle $t_1$ (in blue), $t_2$ (in green), $t_4$ (in orange) and the tetrahedron $T_1$ (in red), that is, it contributes in $4$. 
\item $\deg^{(1,2^*)}_U(\sigma)$ is the number of $2$-simplices (triangles) strictly $2$-upper adjacent to $\sigma$. This is only the triangle $t_3$ (in yellow).  
\item $\deg^{(2,3^*)}_U(\sigma)$ is the number of $3$-simplices (tetrahedra) strictly $3$-upper adjacent to $\sigma$. These are the tetrahedron $T_2$ (in white, formed by the vertices $\{v_0,v_1,v_2,v_3\}$) and the tetrahedron $T_3$ (in purple).   
\end{itemize}
Since $f_i$ denotes the number of $i$-simplices of the simplicial complex, we have that $f_0=13$, $f_1=24$, $f_2=15$ and $f_3=3$ and thus $\sum_{i=0}^{\dim K}f_{i}=55$. Therefore, the maximal simplicial centrality degree centrality of $\sigma$ is:
$$C_{\deg^{*}}(\sigma):=\frac{\deg^*(\sigma)}{\sum_{i=0}^{\dim K}f_{i}-1}=\frac{7}{54}\,.$$

\end{example}

Let us end this subsection by defining the normalised average degree of a simplicial network. Since there are several types of adjacency notions for simplices, we shall define a maximal version of this measure by using both the strict upper degree and the maximal adjacency degree, others definitions can be given in an analogous way.

\begin{definition}\quad 
\begin{itemize}
\item We define the strictly upper normalised average degree of $C_q(K)$ by:
$$\bar E_U^{*}[C_q(K)]:=\frac{\sum_{i=1}^{f_q}\deg_U^*(\sigma_i^{(q)})}{M_{q}}\,,$$
where we are denoting $$M_{q}=\sum_{h=1}^{\dim K-q}f_q\cdot \binom{f_0-(q+1)}{q+h}$$ and $$\deg_U^*(\sigma_i^{(q)})=\sum_{h=1}^{\dim K-q}\deg_U^{h,(q+h)^*}(\sigma_i^{(q)})\,.$$

\item We define the strictly upper normalised average degree of a simplicial complex $K$ by:
$$\bar E_U^{*}[K]:=\frac{\sum_{q=0}^{\dim K}\sum_{i=1}^{f_q}\deg_U^*(\sigma_i^{(q)})}{\sum_{q=0}^{\dim K}M_{q}}\,.$$
\end{itemize}
\end{definition}

\begin{remark}
 The number $M_{q}$ in the denominator reflects the fact that for each chosen $q$-simplex in $C_q(K)$ (there are $f_q$ $q$-simplices in $K$), there are $\binom{f_0-(q+1)}{q+h}$ possible $q+h$-simplices which can be formed with the $q+1$ vertices of the $q$-simplex joint with the $f_0-(q+1)$ remaining ones.
\end{remark}

Using the $p$-adjacency degree for $q$-simplices (with $q>0$), we define the following.
\begin{definition}\quad
\begin{itemize}
\item We define the maximal adjacency normalised average degree of $C_q(K)$ by:
$$\bar E_A^{*}[C_q(K)]:=\frac{\sum_{i=1}^{f_q}\deg_A^*(\sigma_i^{(q)})}{N_{q}}\,,$$
where $$N_{q}=\sum_{p=0}^{q-1}f_q\cdot \binom{q+1}{p+1}\cdot \big(\sum_{q'=p+1}^{\dim K} \binom{f_0-(p+1)}{q'}-1\big)$$ and $$\deg_A^*(\sigma_i^{(q)})=\sum_{p=0}^{q-1}\deg_A^{p^*}(\sigma_i^{(q)})\,.$$ 

\item We define the maximal adjacency normalised average degree of $K$ by:
$$\bar E_A^{*}[K]:=\frac{\sum_{q=0}^{\dim K}\sum_{i=1}^{f_q}\deg_A^*(\sigma_i^{(q)})}{\sum_{q=0}^{\dim K}N_{q}}\,.$$
\end{itemize}
\end{definition}
\begin{remark}
$N_{q}$ reflects the fact that each $q$-simplex (there are $f_q$ $q$-simplices in $K$), has $\binom{q+1}{p+1}$ $p$-faces, and for each of these faces there are $\binom{f_0-(p+1)}{q'}$ possible $q'$-simplices which can be formed with the $p+1$ vertices of the $p$-face joint with the $f_0-(p+1)$ remaining ones.
\end{remark}

Using these measures, we define next the maximal normalised average degree of a simplicial network. 

\begin{definition}\quad 
\begin{itemize}
\item We define the maximal normalised average degree of $C_q(K)$ by:
$$\bar E^{*}[C_q(K)]:=\frac{\sum_{i=1}^{f_q}\deg^*(\sigma_i^{(q)})}{M_q+N_{q}}\,,$$
where  $\deg^*(\sigma_i^{(q)})=\deg_A^{p^*}(\sigma_i^{(q)})+\deg_A^{p^*}(\sigma_i^{(q)})$ is the maximal simplicial degree of Definition \ref{d:simpdeg}. 

\item We define the maximal normalised average degree of $K$ by:
$$\bar E^{*}[K]:=\frac{\sum_{q=0}^{\dim K}\sum_{i=1}^{f_q}\deg^*(\sigma_i^{(q)})}{\sum_{q=0}^{\dim K}(M_q+N_{q})}\,.$$

\item Vertex case. For $q=0$ we define:
$$\bar E^{*}[C_0(K)]:=\bar E_U^{*}[C_0(K)]=\frac{\sum_{i=1}^{f_0}\deg_U^*(v_i)}{M_{0}}\,,$$
where $$M_{0}=\sum_{h=1}^{\dim K}f_0\cdot \binom{f_0-1}{h}$$ and $$\deg_U^*(v_i)=\sum_{h=1}^{\dim K}\deg_U^{h,h^*}(v_i)\,.$$ 
\end{itemize}
\end{definition}

For $q=0$ and $h=1$, the maximal normalised average degree $\bar E^*[K]$ coincides with the normalised average degree of the graph network.

\subsection{Eigenvector Centralities}\quad

This centrality measure is a natural extension of the degree centrality and was first introduced by Bonacich (\cite{Bon72a,Bon72b,Bon07}). The idea behind the  eigenvector centrality is that a node will have a high eigenvector value if the node has high degree and is connected to other nodes that also have high degree.   %a later variant of it is an important part of Google's PageRank algorithm (\cite{LM06}).
In contrast to degree centrality, eigenvector centrality takes into account both the number and the quality of connections to a given node, determining the influence of a node based on that of its neighbours. Thus it considers the entire pattern of the network, which makes this measure suitable for detecting groups of important nodes. In other words, eigenvector centrality is sensitive not only to how well an individual is connected to its neighbours but also to how well these neighbours are connected.  As a result, in a graph a node with few edges to high eigenvector centrality nodes may have a higher eigenvector centrality than a node with more edges to low eigenvector centrality nodes.

For nodes in a graph the eigenvector centrality is obtained from an eigenvector decomposition of the adjacency matrix. Given a graph with vertices $\{v_1,\ldots, v_n\}$ let us  denote by $A=(a_{ij})$ its  adjacency matrix, that is, $a_{ij}=1$ if vertices $v_i$ and $v_j$ are connected by an edge and $0$ otherwise. The eigenvector centrality gives each vertex $v_i$ a score $x_i$ proportional to the sum of the scores of its neighbours: 

\begin{equation}\label{eq:evc}
x_i=\frac{1}{\lambda}\displaystyle\sum_{j=1}^{n}a_{ij}x_j
\end{equation}

In order to determine the desirable value for $\lambda$, note that equation (\ref{eq:evc}) can be rewritten in matrix form as $(A-\lambda I)x=0$, with $x=(x_1,\ldots,x_n)$ and being $I$ the $n$-identity matrix. Since $A$ is symmetric  all its eigenvalues are real, its eigenvectors are orthogonal with respecto to the standard euclidean product and it is diagonalizable. Furthermore, the Perron-Frobenius theorem guarantees that  $A$ has a unique largest real eigenvalue which corresponds to an eigenvector with strictly positive components. So that, uniqueness of this definition is ensured by taking $\lambda$ as such largest eigenvalue of $A$ and defining the vector of eigenvector centralities as the (suitably normalized) eigenvector of $A$ corresponding to its largest eigenvalue.

When dealing instead with phenomena whose process dynamic depends not only on interacting individuals but also on interactions between communities of different sizes, one should consider systems represented by simplicial complexes. In this way, the eigenvector centrality measure was extending to simplicial networks in \cite{ER18}, in terms of the adjacency matrix $A(q)$ associated to the adjacency among $q$-simplices and its corresponding largest eigenvalue. Recall that this matrix is defined by:
\begin{equation}\label{eq:adjmatrix}
A(q)_{ij}=
\begin{cases}
1 & \text{ if } \sigma_i^{(q)}\sim_L\sigma_j^{(q)} \text{ and } \sigma_i^{(q)}\not\sim_U\sigma_j^{(q)}\\
0 & \text{ if } i=j \text{ or }  \sigma_i^{(q)}\not\sim_L\sigma_j^{(q)} \text{ or } \sigma_i^{(q)}\sim_U\sigma_j^{(q)}
\end{cases}
\end{equation}
where  $\sigma_i^{(q)}\sim_L\sigma_j^{(q)}$ if they share a common $(q-1)$-face and $\sigma_i^{(q)}\sim_U\sigma_j^{(q)}$ if  they are both faces of the same common $(q+1)$-simplex.\\
Then, the simplicial eigenvector centrality of the $i$th $q$-simplex in a simplicial complex is given by the $i$th component of the principal eigenvector of the adjacency matrix $A(q)$.

Note that the adjacency matrix $A(q)$ given in Equation \ref{eq:adjmatrix} is $L^a-D$, where $D$ is a diagonal matrix whose entries are the degree of the $q$-simplices and $L^a$ is the resulting matrix considering the $(q,1)$-Laplacian matrix in absolute values. That is, and according to Sections \ref{sec:higheradj} and \ref{s:qhLap}:  
$$
D_{ii}=\deg_U^{1,q+1}(\sigma_i^{(q)})+\deg_L^{1,q-1}(\sigma_i^{(q)})$$
and
$$(L^a)_{ij}=|(L_{q,1,1})_{ij}|\,.$$

It would be tempting to think this still works when one tries to analyze the relevance of simplices which share faces of dimension smaller than $p=q-1$ (that is, $p$-lower adjacent simplices, in terms of Definitions \ref{d:qhLoAdj}). However, as we have pointed out in Remark \ref{rem:Lp}, in contrast to the ordinary lower adjacency, two $p$-lower adjacent simplices could have more than one common lower $p$-face. As a consequence there might be zeros in the multi combinatorial Laplacian corresponding to $p$-lower adjacent simplices but so that the orientations of their common $p$-faces are opposite pairwise and thus it cancels the corresponding oriented degree. On the other hand, since two $p$-lower adjacent simplices with more that one common $p$-face are also $(p+1)$-lower adjacent, one has that the uniqueness of the common lower simplex is ensured for strictly $p$-lower adjacent simplices. Hence, taking into account the notion of $p$-adjacency given in Definition \ref{d:qhAdj} we will be able to extend to higher order the above simplicial centrality measure for $q$-simplices. It is made as follows.

Firstly, by using the notation of the Theorem \ref{thm:degrees}, we know that two $q$-simplices $\sigma_i^{(q)}$ and $\sigma_j^{(q)}$ are $p$-adjacent if and only if $adj^p(\sigma_i^{(q)},\sigma_j^{(q)})=1$. Even more, we have
\begin{equation}
adj^p(\sigma_i^{(q)},\sigma_j^{(q)})=
\begin{cases}
1 & \text{ if } \sigma_i^{(q)}\sim_{A_p}\sigma_j^{(q)}\\
0 & \text{ if } i=j \text{ or }  \sigma_i^{(q)}\not\sim_{A_p}\sigma_j^{(q)}
\end{cases}
\end{equation}
Then the $p$-adjacency matrix for the set $\{\sigma^{(q)}_1\ldots,\sigma^{(q)}_n\}$ of $q$-simplices of a simplicial complex is defined to be the matrix $A(q,p)$ whose $ij$th entry is
$$A(q,p)_{ij}=adj^p(\sigma_i^{(q)},\sigma_j^{(q)})\,.$$
Finally  we define the $p$-adjacency simplicial eigenvector centrality as:

\begin{definition}\label{d:psimpeigen}
Let $\sigma_i^{(q)}$ be a $q$-simplex in a simplicial complex. The $p$-adjacency simplicial eigenvector centrality of $\sigma_i^{(q)}$ is the $i$th component of the normalized eigenvector of the $p$-adjacency matrix $A(q,p)$ corresponding to its largest eigenvalue, where by normalized eigenvector we mean that its components sum to 1.
\end{definition}

Notice that for $p=q-1$ this definition agrees with the one given in \cite{ER18}, which in turn generalizes the eigenvector centrality measure introduced by Bonacich (for $q=0$ the adjacency shall be given by the upper adjacency).

\subsection{Centrality measures associated with walks and distances}\quad 

Given a simplicial complex $K$, the notion of $p$-nearness is given by saying (see \cite{At72,MR12}) that two simplices $\sigma$ and $\sigma'$ are $p$-near if they share a $p$-face. Two simplices $\sigma$ and $\sigma'$ are $p$-connected if there exists a sequence of simplices $\sigma, \sigma(1), \sigma(2), \dots ,\sigma(r),\sigma'$ such that any two consecutive ones share at least a $p$-face (notice that if $\sigma$ and $\sigma'$ are $p$-near then they are also $(p-1)$, $(p-2)$, ... and $0$-near). This implies the use of the $p$-lower adjacency of Definition \ref{d:qhLoAdj} for each consecutive pair of simplices. If we denote $K_p$ the set of simplices in $K$ of dimension greater or equal than $p$, then $p$-connectedness is an equivalence relation on $K_p$ and the quotient space $\bar K_p$ represents the set of equivalence classes of $p$-connected simplices, called the $p$-connected components of $K$.

On the other hand, in \cite{ER18} the authors define a $s^{q}$-walk (for $q>0$) as an alternating sequence of $q$-simplices and $(q-1)$-simplices: $$\sigma_1^{(q)},\tau_1^{(q-1)}\sigma_2^{(q)},\tau_2^{(q-1)},\dots ,\tau_r^{(q-1)}\sigma_{r+1}^{(q)}$$ 
such that, for all $i=1,\dots ,r$, $\tau_i^{(q-1)}$ is a face of both $\sigma_i^{(q)}$ and $\sigma_{i+1}^{(q)}$ and $\sigma_i^{(q)}$ and $\sigma_{i+1}^{(q)}$ are not both faces of the same $(q+1)$-simplex. This definition implies the use of the $(q-1)$-adjacency of Definition \ref{d:qhAdj} for each consecutive pair of $q$-simplices. Following this idea, we give an alternative definition of $p$-nearness which generalises both definitions using the maximal $p$-adjacency of Definition \ref{d:qhAdj}.

\begin{definition}\label{d:pnear}\quad 
\begin{itemize}
\item Two simplices $\sigma^{(q)}$ and $\sigma^{(q')}$ are said to be maximal $p$-near if they are maximal $p$-adjacent, that is,  if $\sigma^{(q)}\sim_{A_{p^*}}\sigma^{(q')}$.

\item A $(p_1,p_2,\dots ,p_r)$-walk between the simplices $\sigma_1^{(q_1)}$ and $\sigma_{r+1}^{(q_{r+1})}$ is a sequence of simplices:
$$\sigma_1^{(q_1)}\tau_1^{(p_1)}\sigma_2^{(q_2)}\tau_2^{(p_2)}\cdots \tau_r^{(p_r)}\sigma_{r+1}^{(q_{r+1})}$$
such that $\sigma_i^{(q_i)}$ and $\sigma_{i+1}^{(q_{i+1})}$ are maximal $p_i$-near, for each $i\in \{1,\dots, r\}$.

\item A $p$-walk between the simplices $\sigma_1^{(q_1)}$ and $\sigma_{r+1}^{(q_{r+1})}$ is a sequence as above:
$$\sigma_1^{(q_1)}\tau_1^{(p_1)}\sigma_2^{(q_2)}\tau_2^{(p_2)}\cdots \tau_r^{(p_r)}\sigma_{r+1}^{(q_{r+1})}$$
where $p=\min \{p_1,\dots ,p_r\}$.

\item We call $r$ the lenght of the $p$-walk and we say that two simplices are maximal $p$-connected if there exists a $p$-walk between them.
\end{itemize}
\end{definition}

\begin{remark}
Notice that if two simplices are maximal $p$-near, they are not maximal $(p-1)$-near. They are clearly $p$-near and thus also $(p-1)$-near, $(p-2)$-near, and so on. 
\end{remark}

\begin{remark}
If $q_i=q$ for all $i$ and $p=q-1$, then the above notion of $p$-walk recovers the definition of a $s^q$-walk given in \cite[Def. 9]{ER18}.
\end{remark}

\begin{figure}[!h]
\centering
\includegraphics[scale=0.9]{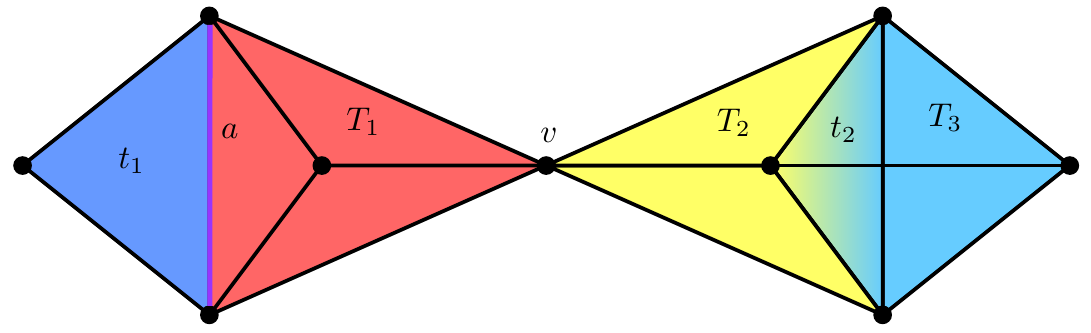}
\caption{Maximal $p$-nearness and walks.}
\label{fig:near}
\end{figure}

\begin{example}
In Figure \ref{fig:near} we have that:
\begin{itemize}
\item The triangle $t_1$ (in blue) and the tetrahedron $T_1$ (in red) are maximal $1$-near (sharing the purple face given by the edge $a$).
\item $T_1$ and the tetrahedron $T_2$ (in yellow) are maximal $0$-near (sharing the vertex $v$).
\item $T_2$ and the tetrahedron $T_3$ (in light blue) are maximal $2$-near (sharing the face given by the triangle $t_2$).
\item There exits a $0$-walk between $t_1$ and $T_3$.
\end{itemize}
\end{example}

Let us point out that with this definition, maximal $p$-connectedness is not an equivalence relation, since a simplex cannot be maximal $p$-adjacent to itself, and thus is not reflexive. Then, we define the following equivalence relation:
$$(\sigma^{(q)},\sigma^{(q')})\in R_p\,\iff \, \begin{cases}\sigma^{(q)} \, \mbox{ and }\, \sigma^{(q')}\, \mbox{ are maximal }p\mbox{-connected }\\ \mbox{ or }\sigma^{(q)}=\sigma^{(q')}
\end{cases}$$

Thus, the quotient space $\bar K_p^*=K_p/R_p$ represents the set of equivalence classes of maximal $p$-connected simplices. We shall refer to these equivalence classes as the maximal $p$-connected components of $K$. 

\begin{remark}
Notice that if $\sigma^{(q')}\sim_{A_{p^*}}\sigma^{(q)}$, then $\sigma^{(q')}\sim_{L_{p^*}}\sigma^{(q)}$, but the converse is no longer true in general. Thus, if $\sigma^{(q')}$ lies in a maximal $p$-connected component, then it also lies in a $p$-connected component, but again, the opposite direction is no longer true in general. 
\end{remark}

The number of elements, $Q_p^*$, of the quotient set $\bar K_p^*$ (that is, $Q_p^*$ is the number of maximal $p$-connected components of $K$) gives a generalisation of the topological invariant $Q$-vector: $(Q_{\dim K},\dots ,Q_1,Q_0)$ (also referred as the first structure vector in the literature). We shall refer to it as the $Q^*$-vector: $(Q^*_{\dim K},\dots ,Q^*_1,Q^*_0)$. 

Once we have walks in a simplicial complex, we can define a new distance. 

\begin{definition}
The $p$-distance between two simplices $\sigma^{(q)}$ and $\sigma^{(q')}$, written $d_p(\sigma^{(q)},\sigma^{(q')})$, is defined to be the smallest length among the $p$-walks between $\sigma^{(q)}$ and $\sigma^{(q')}$. If there are no $p$-walks we put $d_p(\sigma^{(q)},\sigma^{(q')})=\infty$. If $q=q'=0$ we define the distance between vertices as the usual one.
\end{definition}
 
\begin{proposition}
$d_p$ is a generalised\footnote{The term generalised distance will be used to mean a distance function $d\colon X\times X\to [0,\infty]$ on a set $X$ satisfying all the usual metric space axioms except that it need not be finite. Any such functions defines a topology on $X$ in the usual way and induces a metric space structure on each connected component of $X$.} distance on $K$.
\end{proposition}
\begin{proof}
Clearly $d_p(\sigma^{(q)},\sigma^{(q')})\geq 0$ for all $\sigma^{(q)}$ and $\sigma^{(q')}$, and this value is zero if and only if $\sigma^{(q)}=\sigma^{(q')}$. To prove $d_p(\sigma^{(q)},\sigma^{(q')})=d_p(\sigma^{(q')},\sigma^{(q')})$ one just needs to note that any $p$-walk between $\sigma^{(q)}$ and $\sigma^{(q')}$ gives a  $p$-walk between $\sigma^{(q')}$ and $\sigma^{(q)}$, so that $d_p(\sigma^{(q')},\sigma^{(q)})\leq d_p(\sigma^{(q)},\sigma^{(q')})$. A symmetric argument shows that $d_p(\sigma^{(q)},\sigma^{(q')})\leq d_p(\sigma^{(q')},\sigma^{(q')})$, thus $d_p(\sigma^{(q)},\sigma^{(q')})=d_p(\sigma^{(q')},\sigma^{(q')})$. Finally, the inequality 
$$d_p(\sigma^{(q)},\sigma^{(q'')})\leq d_p(\sigma^{(q)},\sigma^{(q')})+d_p(\sigma^{(q')},\sigma^{(q'')})$$ holds true since given a $p$-walk from $\sigma^{(q)}$ to $\sigma^{(q')}$ and a $p$-walk from $\sigma^{(q')}$ to $\sigma^{(q'')}$ one gets a $p$-walk from $\sigma^{(q)}$ to $\sigma^{(q'')}$.
\end{proof}

\begin{definition}
Given $\sigma^{(q)}\in K_p$, we define the $p$-eccentricity of $\sigma^{(q)}$ by:
$$\epsilon_p(\sigma^{(q)}):=\max_{\sigma^{(q')}\in K_p}d_p(\sigma^{(q)},\sigma^{(q')})$$
and the $p$-diameter of $K$ as:
$$D_p:=\max_{\sigma^{(q)}\in K}\epsilon_p(\sigma^{(q)})$$
\end{definition}

Let us denote by $\bar K^*_p=K_p/R_p=\bar K^*_{p,1}\coprod \dots \coprod \bar K^*_{p,Q^*_p}$ the maximal $p$-connected components of $K$, and let us denote by $Q^*_{p,l}$ the number of simplices in the maximal $p$-connected component $\bar K^*_{p,l}$ (for each $l=1,\dots Q^*_p$).
\begin{definition}
We define the average simplicial shortest $p$-walk length by:
$$l_p=\frac{2\sum_{i<j}d_p(\sigma_i^{(q_i)},\sigma_j^{(q_j)})}{Q^*_{p,l}\cdot (Q^*_{p,l}-1)}\,.$$
\end{definition}

Let us end this section by defining new generalisations of two very useful and commonly used centralities in complex networks: the closeness and betweenness centralities. 

\begin{definition}
Given $\sigma_i^{(q_i)}\in K_p$ we define the $p$-closeness centrality by:
$$C_{CL,p}(\sigma_i^{(q_i)}):=\frac{1}{\sum_{i\neq j}d_p(\sigma_i^{(q_i)},\sigma_j^{(q_j)})}$$
\end{definition}
\begin{remark}
We treat $\frac{1}{\infty}=0$. Under this convention, and following the standard literature, the centrality measure should have been referred as the harmonic $p$-closeness centrality.
\end{remark}

\begin{definition}
Given $\sigma_k^{(q_k)}\in \bar K^*_{p,l}\subseteq \bar K^*_p$, we define the $p$-betweenness centrality of a simplex $\sigma_k^{(q_k)}$ as:
$$C_{B,p}(\sigma_k^{(q_k)}):=\frac{2}{(Q^*_{p,l}-1)(Q^*_{p,l}-2)}\cdot \sum_{i\neq k \neq j}\frac{l_{ij,p}(\sigma_k^{(q_k)})}{l_{ij,p}}\,,$$ 
where $l_{ij,p}$ stands for the total number of shortest $p$-walks between $\sigma_i^{(q_i)}$ and $\sigma_j^{(q_j)}$, and $l_{ij,p}(\sigma_k^{(q_k)})$ denotes the total number of shortest $p$-walks between $\sigma_i^{(q_i)}$ and $\sigma_j^{(q_j)}$ passing through $\sigma_k^{(q_k)}$. 
\end{definition}

\begin{remark}
The term $\frac{(Q^*_{p,l}-1)(Q^*_{p,l}-2)}{2}$ is the number of pairs of simplices, different from $\sigma_k^{(q_k)}$, lying in the same maximal $p$-connected component $\bar K^*_{p,l}$ as $\sigma_k^{(q_k)}$.
\end{remark}

\subsection{Simplicial clustering: a centrality measure associated with both the higher order degree and the $p$-distance}\quad 

Let us finish this subsection by defining a clustering coefficient for $q$-simplices. In \cite{MRV08} a clustering coefficient of a $q$-simplex $\sigma^{(q)}$ is defined as follows:
\small{$$\frac{\mbox{number of faces that neighbour simplices to }\sigma^{(q)}\mbox{ share between themselves}}{\frac{z-1}{2}(\mbox{total number of faces that neighbour simplices can share between themselves)}}\,,$$}where $z$ denotes the number of neighbour simplices of $\sigma^{(q)}$. Since the number of $p$-faces of a $q'_i$-simplex  is given by the combinatorial number $\binom{q'_i+1}{p+1}$, which by equation (\ref{e:lowfaces}) is equal to de lower degree $\deg_L^{h,p}(\sigma^{(q'_i)})$, and the number $z$ of neighbour simplices of $\sigma^{(q)}$ can be defined with the upper degree of $\sigma^{(q)}$ of Definition \ref{d:hpUdeg}, then the clustering coefficient of \cite{MRV08} could be also written down with our definitions. Nice and closed formulas are also contained in that paper to explicitly compute this clustering coefficient. But let us point out that with this definition, if we have a triangle having two tetrahedra attached to a common vertex, then its clustering coefficient is non zero since the two  tetrahedra share a $0$-face (the vertex), but nonetheless one could think that this two tetrahedron communities are neighbours of the triangle but are not linked between themselves by a different agent to that of the common vertex, so that one might state that the clustering coefficient of the triangle should be zero. 

With this idea, instead of rewriting the clustering coefficient of \cite{MRV08} in terms of the higher order lower and upper degrees, we will give a different definition of simplicial clustering of a simplex, which generalises the standard clustering coefficient of a vertex by using both the definition of maximal $p$-adjacency degree and the strictly upper degree of a simplex.

In a graph network, the clustering coefficient of a vertex $v$ is defined as the ratio:
$$C(v)=\frac{\mbox{number of links among its neighbour vertices}}{\mbox{number of all possible links among its neighbour vertices}}\,,$$
that is:
$$C(v)=\frac{\mbox{number of links among its neighbour vertices}}{\frac{\deg(v)\cdot (\deg(v)-1)}{2}}\,,$$

In order to generalise this definition for a simplex in a simplicial network, let us start with the vertex case ($0$-simplex), and to explain what we mean by a neighbour of a $0$-simplex and what we allow as a link between neighbours of a $0$-simplex. 

\begin{definition}
We say that a $h$-simplex $\sigma^{(h)}$ is a maximal neighbour of a $0$-simplex $v$ if it is strictly $h$-upper adjacent to $v$, that is:
$$\sigma^{(h)} \mbox{ is a maximal neighbour of }v\,\iff \, \sigma^{(h)}\sim_{U_{h^*}}v\,.$$
\end{definition}

\begin{definition}
Let $\sigma^{(h)}$ and $\sigma^{(\bar h)}$ be two maximal neighbour simplices to the $0$-simplex $v$. We say that $\sigma^{(h)}$ and $\sigma^{(\bar h)}$ are linked if they share a $p$-face which is different from $v$.
\end{definition}

\begin{remark}
A maximal neighbour of a vertex $v$ is a maximal collaborative simplicial community to which the vertex $v$ belongs to. Maximal neighbour simplices of a vertex $v$ are maximal collaborative communities (to which the vertex belongs to) sharing a collaborative sub-community ($p$-face) which contains more vertices than the vertex $v$ itself.
\end{remark}

Let us now define a simplicial clustering coefficient for a $0$-simplex in a simplicial network.

\begin{definition}
The simplicial clustering coefficient of a $0$-simplex $v$ in $K$ is defined as the ratio:
$$C_S(v)=\frac{\mbox{number of links among its maximal neighbour simplices}}{\frac{\deg^*_U(v)\cdot(\deg^*_U(v)-1)}{2}}\,,$$
where $\deg^*_U(v):=\sum_{h=1}^{\dim K}\deg^{(h,h^*)}_U(v)$.
\end{definition}

\begin{figure}[!h]
\centering
\includegraphics[scale=0.7]{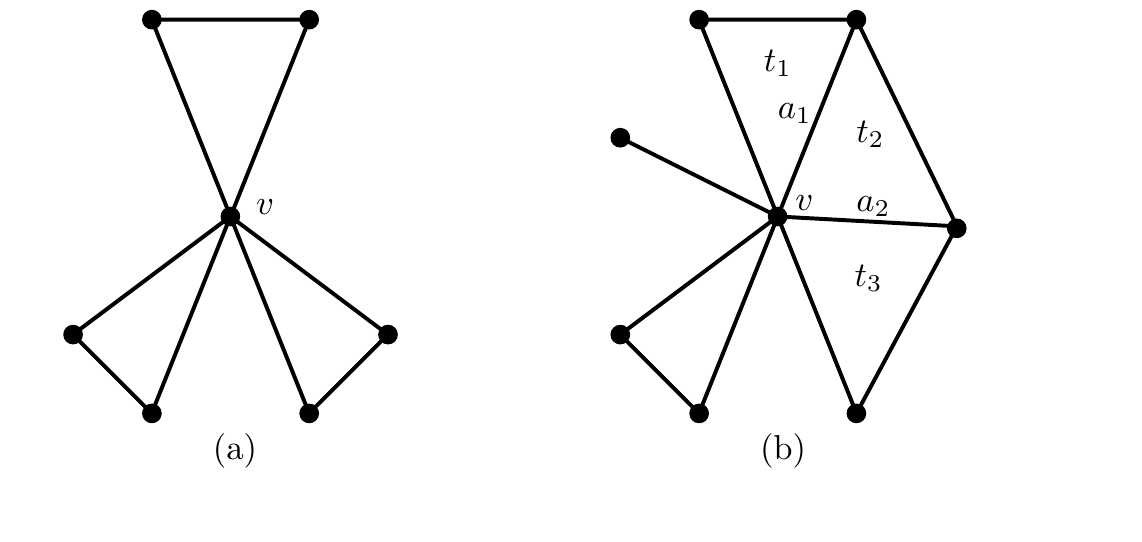}
\caption{Simplicial clustering coefficient of a vertex.}
\label{fig:SCv}
\end{figure}

\begin{example}
In Figure \ref{fig:SCv} (a) the graph clustering coefficient of the vertex $v$ is: 
$$C(v)=\frac{3}{\frac{\deg (v)\cdot (\deg (v)-1)}{2}}=\frac{3}{\frac{6\cdot 5}{2}}=\frac{1}{5}\,.$$
Nonetheless, its simplicial clustering coefficient is:
$$C_S(v)=\frac{0}{\frac{\deg^*_U(v)\cdot(\deg^*_U(v)-1)}{2}}=\frac{0}{\frac{3\cdot 2}{2}}=\frac{0}{3}=0$$
since $\deg^*_U(v)=\deg_U^{(1,1^*)}(v)+\deg_U^{(2,2^*)}(v)=0+3=3$ and $\deg_U^{(h,h^*)}(v)$ is the number of $h$-simplices strictly $h$-upper adjacent to $v$. Notice that none of the triangles which are upper adjacent to $v$ are linked by a face different from $v$, and this is why the numerator is $0$.

In Figure \ref{fig:SCv} (b) the graph clustering coefficient of the vertex $v$ is: 
$$C(v)=\frac{3}{\frac{\deg (v)\cdot (\deg (v)-1)}{2}}=\frac{4}{\frac{7\cdot 6}{2}}=\frac{4}{21}\,,$$
whereas the simplicial clustering coefficient of $v$ is:
$$C_S(v)=\frac{2}{\frac{\deg^*_U(v)\cdot(\deg^*_U(v)-1)}{2}}=\frac{2}{\frac{5\cdot 4}{2}}=\frac{1}{5}\,,$$
where $\deg^*_U(v)=\deg_U^{(1,1^*)}(v)+\deg_U^{(2,2^*)}(v)=1+4=5$, since: the strict upper degree $\deg_U^{(1,1^*)}(v)$ is the number of edges to which $v$ belongs to and such that they are not contained in a triangle (so there is only one); and the strict upper degree $\deg_U^{(2,2^*)}(v)$ is the number of triangles to which $v$ belongs to and such that they are not contained in a tetrahedron (thus, there are $4$, the triangles $t_1$, $t_2$ and $t_3$ and the one not marked). 

Since the triangles $t_1$ and $t_2$ are linked by the face $a_1$, and the triangles $t_2$ and $t_3$ are linked by the face $a_3$, we have that there exist $2$ links among the maximal neighbours of $v$ given by $a_1$ and $a_2$.

\end{example}

Let us consider the $q$-simplex case for $q>0$. In this situation, there exists lower adjacency and thus, the notion of maximal $p$-adjacency joint with the strict upper adjacency will suit for us. The reason is that, if one has an edge ($1$-simplex), there might be maximal simplicial communities which are maximal $0$-adjacent to the edge (lower adjacent to the edge in a vertex of the edge, but not upper adjacent to the edge), but there might also exist maximal simplicial communities on which our edge is nested in (that is, strict upper adjacent simplices to the edge). Thus, we have to consider both types as maximal neighbours.

\begin{definition}\label{d:maxneigh}
We say that a $q'$-simplex $\sigma^{(q')}$ is a maximal neighbour of a $q$-simplex $\sigma^{(q)}$ if either:
\begin{itemize}
\item there exists $p<q$ such that $\sigma^{(q')}$ is maximal $p$-adjacent to $\sigma^{(q)}$, that is $\sigma^{(q')}\sim_{A_{p^*}}\sigma^{(q)}$, or

\item $\sigma^{(q')}$ is strictly $q'$-upper adjacent to $\sigma^{(q)}$, that is $\sigma^{(q')}\sim_{U_{q'^*}}\sigma^{(q)}$.
\end{itemize}
\end{definition}

\begin{definition}\label{d:linkmax}
Let $\sigma^{(q'_i)}$ and $\sigma^{(q'_j)}$be two maximal neighbour simplices to $\sigma^{(q)}$. We say that $\sigma^{(q'_i)}$ and $\sigma^{(q'_j)}$ are linked if either:
\begin{itemize}
\item they share a face which is different from $\sigma^{(q)}$ and all of its faces, or
\item there exists a $0$-walk between them of $0$-distance $d_0(\sigma^{(q_i')},\sigma^{(q_j')})=2$, and such that it does not contain any of the faces of $\sigma^{(q)}$.
\end{itemize}
\end{definition}

\begin{remark}
The first condition in the above definition means that if the two maximal neighbour simplices to $\sigma^{(q)}$ share a face, then this face has to contain at least one vertex different from those vertices of $\sigma^{(q)}$. The second consists of allowing maximal neighbour simplices to $\sigma^{(q)}$, which do not share a face, to be connected by at least an edge which is not contained in $\sigma^{(q)}$. If there were more than one $0$-walk between two maximal neighbour simplices satisfying the second condition of the above definition, we will only count one link. 
\end{remark}

Let us now define a clustering coefficient for a simplex in a simplicial network.

\begin{definition}\label{d:simpclust}
The simplicial clustering coefficient of a $q$-simplex $\sigma^{(q)}$ in $K$ (with $q>0$) is defined as the ratio:
$$C_S(\sigma^{(q)})=\frac{\mbox{number of links among its maximal neighbour simplices}}{\frac{\deg^*(\sigma^{(q)})\cdot (\deg^*(\sigma^{(q)})-1)}{2}}\,,$$
where $\deg^*(\sigma^{(q)})=\deg^*_A(\sigma^{(q)})+\deg^*_U(\sigma^{(q)})$ is the maximal  simplicial degree of $\sigma^{(q)}$ given in Definition \ref{d:simpdeg}.
\end{definition}

\begin{remark}
Notice that we can compute the simplicial clustering coefficient of a simplex in a simplicial network by using the results of the previous section.
\end{remark}

This definition generalises the graph clustering definition for vertices in the following sense. 
\begin{enumerate}
\item Choose a $q$-simplex (the one wanted to know its simplicial clustering) and draw it as vertex $v$ in a new graph. 
\item Represent also as vertices of the new graph all its maximal neighbours simplices. 
\item Draw an edge from $v$ to other vertices representing maximal neighbour simplices to $\sigma^{(q)}$.
\item Draw also an edge among the neighbour vertices to $v$ whose associated maximal neighbour simplices are linked.
\end{enumerate}
Then, the clustering coefficient of $\sigma^{(q)}$ here defined is precisely the standard graph clustering coefficient of its associated vertex $v$ in the new graph.

\begin{figure}[!h]
\centering
\includegraphics[scale=1.1]{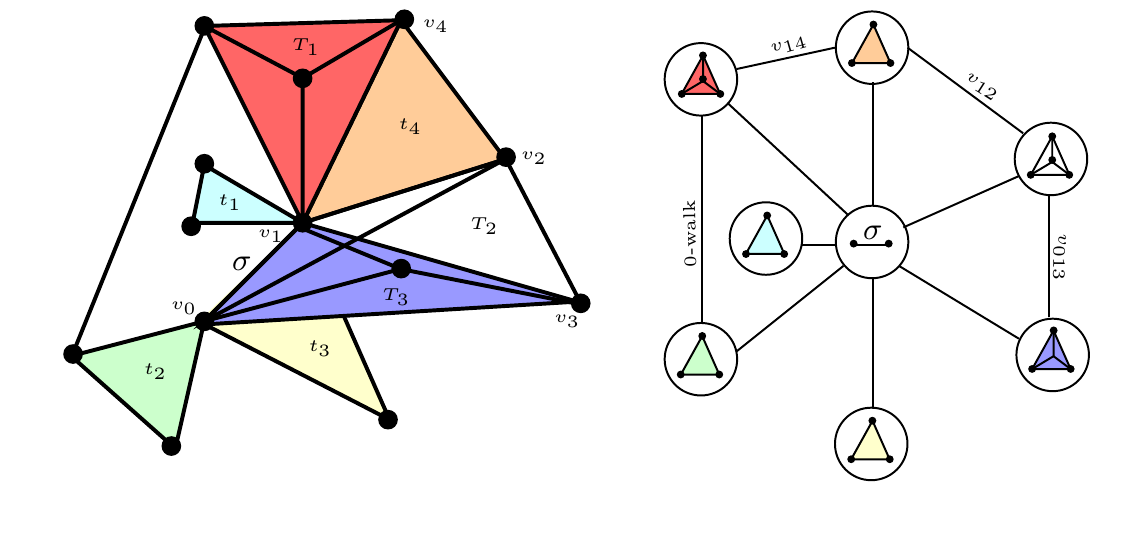}
\caption{Simplicial clustering coefficient of a $1$-simplex.}
\label{fig:SCq}
\end{figure}

\begin{example}
Let us compute the simplicial clustering coefficient of the $1$-simplex $\sigma$ defined by the vertices $v_0$ and $v_1$ in Figure \ref{fig:SCq} left. The general simplicial degree of $\sigma$ is computed in Example \ref{ex:AdjDeg}:
$$
\deg^*(\sigma)=\deg^*_A(\sigma)+\deg^*_U(\sigma)=7\,.
$$

Following Definition \ref{d:maxneigh}, the maximal neighbours of $\sigma$ are the traingles $t_1$, $t_2$ and $t_4$, the tetrahedron $T_1$ (all of them being maximal $0$-adjacent to $\sigma$), and the triangle $t_3$ and the tetrahedra $T_2$ and $T_3$ (being strictly upper adjacent to $\sigma$), Using Definition \ref{d:linkmax}, the links among the maximal neighbours of $\sigma$ are $4$ since:
\begin{itemize}
\item There exists a $0$-walk of $0$-distance $2$ between $T_1$ and $t_2$. 

\item There exists a common face between $T_1$ and $t_4$ given by the edge $v_{14}$ defined by the vertices $\{v_1,v_4\}$.
\item There exists a common face between $t_4$ and $T_2$ given by the edge $v_{12}=\{v_1,v_2\}$.
\item There exists a common face between $T_2$ and $T_3$ given by the triangle $v_{013}=\{v_0,v_1,v_3\}$.
\end{itemize}
Thus, the simplicial clustering coefficient of $\sigma$ is:

$$C_S(\sigma^{(q)})=\frac{\mbox{number of links among its maximal neighbour simplices}}{\frac{\deg^*(\sigma)\cdot (\deg^*(\sigma)-1)}{2}}=\frac{4}{21}\,.$$

Notice that this agrees with the standard graph clustering of $\sigma$ thought of as a vertex in Figure \ref{fig:SCq} right.
\end{example}

\begin{remark}
Depending on what kind of clustering we want to measure in our simplicial network, Definition \ref{d:simpclust} might be generalised by introducing a different notion of link using distinct $s$-distances quantities (also varying the parameter $s$, that is, the allowed $s$-walks). Notice also that when two maximal neighbour simplices to a given one are linked by sharing a face, no matters the dimension of the face is, we just understand it as a connection between them. Thus, weights might be needed for different applications. 
\end{remark}

\section{Conclusions}\label{s:conc}

We have used the recently introduced higher order degree definitions of a simplical community  given in \cite{HHS20} to propose new centrality measures in simplicial complex networks. The centrality measures based on the simplicial degree here presented would allow to understand the importance of simplicial communities in these type of networks and might be useful to understand topological and dynamical properties of complex systems. 
Some of their usefulness are revealed in \cite{HHS20}, where a simplicial connectivity study of several real-world datasets (which include coauthor networks, cosponsoring congress bills, contacts in schools, drug abuse warning networks, e-mail networks or publications and users in online forums) is performed by studying statistical properties and simplicial degree distributions, showing a rich and varied higher-order connectivity structures. We shall study its properties and propose configuration models in future works (following the standard results for graph networks of \cite{BA16} and some simplicial results for $d$-pure simplicial networks of \cite{BR02,BA16}). In addition, we have defined a eigenvector centrality measure for $q$-simplices depending on the higher order adjacency between them. It is based on the idea that a relationship to a more $p$-adjacent $q$-simplex contributes to the own centrality to a greater extent than a relationship to a less well interconnected $q$-simplex. Therefore this centrality measure is expected to prove useful for applications in simplicial networks which require to identify relevant groups, with respect to a fixed order of adjacency, among simplicial communities of the same size.

We have introduced the notions of walks and distances in simplicial complexes, which have allowed us to generalise, to the simplicial case, the closeness and betweenness centralities. We believe that a way is opened to study how information flows in a simplicial network through its simplicial communities. 

We have also define a simplicial clustering coefficient in a network, different from those already known, which generalises the standard graph clustering of a vertex in a graph network, and allows to know not only the simplicial clustering around a node, but the clustering around a simplicial community. Thus, this simplicial clustering could be a useful tool to understand the topology of a simplicial network in terms of the density of its simplicial communities, and it would be interesting to relate it and/or compare it with the first two betti numbers of the simplicial network (usually denoted $\beta_0$ and $\beta_1$),  following the ideas of \cite{BK19}. Moreover, we believe that it might help to understand how cliques and simplicial communities evolve in a simplicial network and to study analogues to clique percolation techniques in simplicial networks.

Finally, let us point out that very recently the contagion phenomena (which include the information and rumor diffusion, epidemic spreading and computer virus transmission) is starting to be studied in the context of higher order interactions among constituents (see \cite{JJK19} for the hypergraphs case or \cite{IPBL19} for the simplicial complex case), and thus it would be interesting to adapt the centrality indices studied, for example the improved gravitational centrality of \cite{WLX18}, to the simplicial case in order to study the spreading capability in simplicial complex systems.

\end{document}